\documentclass[12pt]{amsart}
\usepackage{amsmath}
\usepackage{amsfonts, amssymb, euscript, mathrsfs}
\usepackage{amsthm, upref}
\usepackage{graphicx}
\usepackage[usenames, dvipsnames]{color}

\usepackage{tikz}
\usepackage{enumerate}
\usepackage[margin=0.8in]{geometry}
\usepackage{aliascnt}
\usepackage{hyperref}
\usepackage{verbatim}

\allowdisplaybreaks[4]

\newcommand{\MZ}{\mathbb{Z}}

\newcommand{\BR}{\mathbb{R}}

\newcommand{\SL}{\sum\limits}

\newcommand{\al}{\alpha}
\newcommand{\be}{\beta}

\newcommand{\de}{\delta}

\newcommand{\La}{\Lambda}

\newcommand{\CF}{\mathcal F}

\newcommand{\MP}{\mathbf P}

\newcommand{\MQ}{\mathbb Q}

\newcommand{\Oa}{\Omega}
\newcommand{\oa}{\omega}
\newcommand{\si}{\sigma}

\renewcommand{\phi}{\varphi}

\newcommand{\eps}{\varepsilon}
\newcommand{\la}{\lambda}

\newcommand{\Ra}{\Rightarrow}

\newcommand{\ol}{\overline}

\newcommand{\Tau}{\mathcal T}

\renewcommand{\comment}[1]{}

\newcommand{\md}{\mathrm{d}}

\newcommand{\mP}{\mathbf{p}}

\DeclareMathOperator{\Var}{Var}

\DeclareMathOperator{\diag}{diag}

\DeclareMathOperator{\mes}{mes}

\DeclareMathOperator{\Exp}{Exp}

\begin{document}

\theoremstyle{plain}
\newtheorem{thm}{Theorem}[section]
\newtheorem*{thmnonumber}{Theorem}
\newtheorem{lemma}[thm]{Lemma}
\newtheorem{prop}[thm]{Proposition}
\newtheorem{cor}[thm]{Corollary}
\newtheorem{open}[thm]{Open Problem}
\newtheorem{conj}[thm]{Conjecture}

\theoremstyle{definition}
\newtheorem{defn}{Definition}
\newtheorem{asmp}{Assumption}
\newtheorem{notn}{Notation}
\newtheorem{prb}{Problem}

\theoremstyle{remark}
\newtheorem{rmk}{Remark}
\newtheorem{exm}{Example}
\newtheorem{clm}{Claim}

\author{Andrey Sarantsev}

\title[Two-Sided Infinite Systems]{Two-Sided Infinite Systems\\ of Competing Brownian Particles} 

\address{Department of Statistics and Applied Probability, University of California, Santa Barbara}

\email{sarantsev@pstat.ucsb.edu}

\date{June 13, 2017. Version 29}

\keywords{Competing Brownian particles, gap process, weak convergence, stationary distribution, named particles, ranked particles, stochastic domination, interacting particle systems}

\subjclass[2010]{Primary 60J60, secondary 60J55, 60J65, 60H10, 60K35}

\begin{abstract} Two-sided infinite systems of Brownian particles with rank-dependent dynamics, indexed by all integers, exhibit  different properties from their one-sided infinite counterparts, indexed by positive integers, and from finite systems. Consider the gap process, which is formed by spacings between adjacent particles. In stark contrast with finite and one-sided infinite systems, two-sided infinite systems can have one- or two-parameter family of stationary gap distributions, or the gap process weakly converging to zero as time goes to infinity. 
\end{abstract}

\maketitle

\section{Introduction}

\thispagestyle{empty}

\subsection{Definitions}
The article is devoted to systems of Brownian particles on the real line:
$$
X = (X_n)_{n \in \MZ},\ X_n = (X_n(t),\, t \ge 0),\ n \in \MZ,
$$
which evolve according to the following rule: The dynamics of each particle (more precisely, its drift and diffusion coefficients) depend on its current rank relative to other particles. These systems are called {\it two-sided infinite systems of competing Brownian particles}. Let us define them formally. 

\smallskip

A vector $x \in (x_n)_{n \in \MZ} \in \BR^{\MZ}$ is called {\it rankable} if there exists a bijection $\mP : \MZ \to \MZ$ such that
\begin{equation}
\label{eq:ranking}
x_{\mP(k)} \le x_{\mP(l)}\ \ \mbox{for}\ \ k \le l.
\end{equation}

The following  counterexample shows that not all sequences in $\BR^{\MZ}$ are rankable:
$$
x = (x_n)_{n \in \MZ},\ \ 
x_n = 
\begin{cases}
n^{-1},\ n \ne 0;\\
0,\ n = 0.
\end{cases}
$$

However, if $x \in \BR^{\MZ}$ is rankable, then we can find a bijection $\mP : \MZ \to \MZ$ which satisfies~\eqref{eq:ranking} and resolves ties in lexicographic order: if $x_{\mP(k)} = x_{\mP(l)}$, but $k < l$, then $\mP(k) < \mP(l)$. This is called a {\it ranking permutation} for the vector $x$. Such a permutation is unique up to a shift: For any two ranking permutations $\mP$ and $\mP'$, there exists an $m \in \MZ$ such that $\mP(k) = \mP'(k+m)$ for all $k \in \MZ$. Suppose we fixed a ranking permutation $\mP$ for the vector $x \in \BR^{\MZ}$. For each $i \in \MZ$, the integer $k = \mP^{-1}(i)$ is called the {\it rank} of the component $x_i$. 

\smallskip

We operate in the standard setting: a filtered probability space $(\Oa, \CF, (\CF_t)_{t \ge 0}, \MP)$ with the filtration satisfying the usual conditions. Fix parameters $g_n \in \BR$ and $\si_n > 0$, $n \in \MZ$. Take i.i.d. one-dimensional $(\CF_t)_{t \ge 0}$-Brownian motions $W_n = (W_n(t),\, t \ge 0),\ n \in \MZ$. 

\begin{defn} An infinite family $X = (X_n)_{n \in \MZ}$ of continuous adapted real-valued processes 
$$
X_n = (X_n(t), t \ge 0),\, n \in \MZ,
$$
forms a {\it two-sided infinite system of competing Brownian particles} with {\it drift coefficients} $g_n,\ n \in \MZ$, and {\it diffusion coefficients} $\si_n^2,\ n \in \MZ$, if the following conditions hold true:

\medskip

(a) the vector $X(t) = (X_n(t))_{n \in \MZ}$ is rankable for every $t \ge 0$;

\medskip

(b) for every $t \ge 0$, we can choose a ranking permutation $\mP_{t}$ of $X(t)$, so that for every $k \in \MZ$, the process $(\mP_t(k),\, t \ge 0)$, is $(\CF_t)_{t \ge 0}$-adapted; and the process $t \mapsto X_{\mP_t(k)}(t)$, is a.s. continuous;

\medskip

(c) the components $X_n,\ n \in \MZ$, satisfy the following system of SDEs:
\begin{equation}
\label{eq:SDE-main}
\md X_n(t) = \SL_{k \in \MZ}1\left(\mP_t(k) = n\right)\left(g_k\,\md t + \si_k\, \md W_n(t)\right),\ n \in \MZ.
\end{equation}
Each process $X_n$ is called the {\it $n$th named particle}, with {\it name} $n$. Each process $Y_k = (Y_k(t), t \ge 0)$, defined by $Y_k(t) := X_{\mP_{t}(k)}(t)$, $t \ge 0$, is called the {\it $k$th ranked particle}, with {\it rank} $k$. By construction, ranked particles satisfy $Y_k(t) \le Y_{k+1}(t)$ for all $k \in \MZ,\, t \ge 0$. The processes $W_n,\, n \in \MZ$, are called {\it driving Brownian motions} for this system $X$. 
\label{defn:CBP}
\end{defn}

Loosely speaking, in this system each particle moves as a Brownian motion with drift coefficient $g_k$ and diffusion coefficient $\si_k^2$, as long as it has rank $k$. When particles collide, they might exchange ranks, and in this case they exchange their rank-dependent drift and diffusion coefficients. 

\smallskip

The property (b) is necessary to ensure that particles $X_n,\, n \in \MZ$, can change ranks only when they collide with other particles $X_m,\, m \in \MZ$; or, equivalently, ranked particles $Y_k,\, k \in \MZ$, can change names only when they collide with other ranked particles. 

\smallskip

We can define similar finite systems $(X_n)_{1 \le n \le N}$ of $N$ particles, introduced in \cite{BFK2005}. These systems are also governed by the equation~\eqref{eq:SDE-main}, with the sum over $k = 1, \ldots, N$, instead of over $k \in \MZ$. As in Definition~\ref{defn:CBP}, we denote the $k$th ranked particle at time $t$ by $Y_k(t)$, for $k = 1, \ldots, N$. These ranked particles satisfy 
$$
Y_1(t) \le Y_2(t) \le \ldots \le Y_N(t).
$$
We can also define one-sided infinite systems $(X_n)_{n \ge 1}$, where particles are ranked from bottom to top. These systems were introduced in \cite{IKS2013}. They are governed by~\eqref{eq:SDE-main}, with the sum over $k = 1, 2, \ldots$ instead of over $k \in \MZ$. Here, the ranked particles $Y_k,\, k \ge 1$, satisfy
$$
Y_1(t) \le Y_2(t) \le \ldots
$$
For finite and one-sided infinite systems, we do not have to impose condition (b) from Definition~\ref{defn:CBP}. Rather, we can just rank particles from bottom to top: If we start from assigning rank $1$ to the lowest particles, then such ranking (resolving ties in lexicographic order) is unique, and automatically satisfies the condition (b) above (with $k = 1, 2, \ldots$ instead of $k \in \MZ$). 

\smallskip

Sometimes it is convenient to index particles $X_n$ and $Y_k$ in finite systems from $M$ to $N$, and in one-sided infinite systems from $M$ to $\infty$. We shall sometimes use this alternative indexing in this paper, when we prove our results. In this case, we always indicate that we are using this alternative indexing instead of the standard one. 

\begin{rmk} For a finite, one- or two-sided infinite system, we say {\it initial conditions are ranked} if $X_k(0) = Y_k(0)$ for all $k$. 
\end{rmk}


\begin{defn} For finite, one- and two-sided infinite systems, the {\it gap process} is defined as follows:
$$
Z = (Z(t), t \ge 0),\ \ Z(t) = (Z_n(t)),\ \ Z_n(t) := Y_{n+1}(t) - Y_n(t).
$$
\end{defn}

In other words, the component $Z_n$ is defined as the spacing between adjacent ranked particles $Y_n$ and $Y_{n+1}$. Let $\BR_+ := [0, \infty)$.  The gap process $Z$ takes values:

\smallskip

(a) in the positive orthant $\BR_+^{N-1}$ for a system of $N$ particles;

\smallskip

(b) in $\BR_+^{\infty}$ for a one-sided infinite system;

\smallskip

(c) in $\BR_+^{\MZ}$ for a two-sided infinite system.  

\begin{defn} A {\it stationary gap distribution} (for finite, one- or two-sided infinite systems) is defined as a probability measure $\pi$ in the orthant (finite- or infinite-dimensional) such that there exists a version of the system with $Z(t) \sim \pi$ for all $t \ge 0$. 
\end{defn}

We study two main topics in this article for two-sided infinite systems: (a) weak existence and uniqueness in law; (b) stationary gap distributions and long-term behavior for the gap process $Z(t)$, that is, weak limits of $Z(t)$ as $t \to \infty$. Most of our results in (b) are for the case $\si_n = 1$ for all $n \in \MZ$. 

\subsection{Notation} The symbol $\Ra$ denotes weak convergence. For $\al > 0$, $\Exp(\al)$ stands for the exponential distribution with rate $\al$, and mean $\al^{-1}$.  For $x \in \BR^{\MZ}$, we define 
$$
[x, \infty) := \{y \in \BR^{\MZ}\mid y_i \ge x_i\ \forall i \in \MZ\}.
$$
Take two probability measures $\nu_1$ and $\nu_2$ on $\BR^{\MZ}$. Then $\nu_1$ is {\it stochastically dominated by} $\nu_2$ if 
$$
\nu_1[x, \infty) \le \nu_2[x, \infty)\ \mbox{for all}\ x \in \BR^{\MZ}.
$$
We denote this by $\nu_1 \preceq \nu_2$. Same definition applies to probability measures on $\BR^{\infty}$ and $\BR^N$ for finite $N$. Two random variables $\xi_1$ and $\xi_2$ satisfy $\xi_1 \preceq \xi_2$ if their distributions $P_1$ and $P_2$ satisfy $P_1 \preceq P_2$. Take subsets $I \subseteq J \subseteq \MZ$. For $a = (a_i)_{i \in J} \in \BR^J$, define $[a]_I := (a_i)_{i \in I}$. For a probability measure $\rho$ on $\BR^J$, let $[\rho]_I$ be its marginal, corresponding to the components indexed by $i \in I$: 
$$
(z_i)_{i \in J} \sim \rho\ \mbox{implies}\ [z]_I := (z_i)_{i \in I} \sim [\rho]_I.
$$
Denote the tail of the standard normal distribution by
$$
\Psi(u) = \left(2\pi\right)^{-1/2}\int_u^{\infty}e^{-z^2/2}\,\md z
$$
Fix a $T > 0$. The {\it modulus of continuity} of a function $f : [0, T] \to \BR$, corresponding to $\de > 0$, is defined as
$$
\oa(f, [0, T], \de) := \sup\limits_{\substack{s, t \in [0, T]\\ |t - s| \le \de}}|f(t) - f(s)|. 
$$
The Dirac delta measure at $x$ is denoted by $\de_x$. The symbol $\mathbf{0}$ denotes the origin in $\BR^{\MZ}$. 

\subsection{Comparison with known results} We present some known results on existence and uniqueness, as well on the gap process, for finite and one-sided infinite systems. Then we highlight differences between these results and our new results in this paper for two-sided infinite systems.

\subsubsection{Existence and uniqueness} For finite systems, weak existence and uniqueness in law simply follows from \cite{BassPardoux}. It holds for any values of parameters $g_k \in \BR,\, \si_k > 0,\, k = 1, \ldots, N$, and for any initial condition. For one-sided infinite systems, we need to impose certain assumptions on $g_k,\, \si_k,\, k \ge 1$, as well as on the initial conditions $X(0) = x$, see \cite{S2011, IKS2013}, \cite[Theorem 3.1, Theorem 3.2]{MyOwn6}. The main idea behind the proof of weak existence and uniqueness in law for one-sided infinite systems is as follows: On a finite time interval, a given particle behaves as if it were only in a finite system of particles. Theorem~\ref{thm:existence} below states weak existence and uniqueness in law for two-sided infinite systems.  The proof is quite similar to the case of one-sided infinite systems. 


\subsubsection{Approximation by finite systems} In the paper \cite{MyOwn6}, we have proved that a one-sided infinite system is a weak limit of finite systems, as the number of particles in these finite systems goes to infinity. This result is used to study the gap process. Two-sided infinite systems can also be obtained as weak limits of finite systems, see Lemma~\ref{lemma:approx}. However, the proof for two-sided systems is much more complicated than for one-sided systems, because there is no bottom-ranked  particle in two-sided infinite systems.

\subsubsection{Gap process for finite systems} Consdier a system $X = (X_n)_{1 \le n \le N}$ of $N$ particles. Denote by $\ol{g}_N$  the average of all $N$ drift coefficients: $\ol{g}_N := (g_1 + \ldots + g_N)/N$. Impose  the following stability condition on drift coefficients:
\begin{equation}
\label{eq:stability-intro}
g_1 + \ldots + g_n > n\ol{g}_N,\ \mbox{for}\ n = 1, \ldots, N-1.\ 
\end{equation}
In words, condition~\eqref{eq:stability-intro} means that the average of drift coefficients for a few consecutive lower-ranked particles is larger than the average of all $N$ drift coefficients. It is known from~\cite{5people, PP2008, MyOwn6} that, under condition~\eqref{eq:stability-intro}, there is a unique stationary gap distribution $\pi$. Moreover, $Z(t) \Ra \pi$ as $t \to \infty$, regardless of the initial distribution of $Z(0)$. If condition~\eqref{eq:stability-intro} does not hold, then there are no stationary gap distributions for this finite system. If condition~\eqref{eq:stability-intro} holds together with $\si_n = 1$ for all $n$, then this distribution $\pi$ has an explicit product-of-exponentials form, see \cite{5people}:  
\begin{equation}
\label{eq:product-form}
\pi = \bigotimes\limits_{n=1}^{N-1}\Exp\left(\mu_n\right),\ \ \mu_n := 2(g_1 + \ldots + g_n - n\ol{g}_N),\ n = 1, \ldots, N-1.
\end{equation}
For general $\si_n,\, 1 \le n \le N$, an explicit form of $\pi$ is not known. 

\subsubsection{Gap process for one-sided infinite systems} Consider a system $X = (X_n)_{n \ge 1}$. Assume that $\si_n = 1$ for all $n \in \MZ$, and $\sup|g_n| < \infty$. It was shown in \cite{MyOwn13} that we always have a one-parameter product-of-exponentials family of stationary gap distributions $\pi_a,\, a \in \BR$. In  contrast with finite systems, we do not need to impose any stability condition similar to~\eqref{eq:stability-intro}. Therefore, the weak limit of $Z(t)$ as $t \to \infty$ depends on the initial distribution of $Z(0)$. For certain cases, we can describe there weak limits for at least some initial distributions, see \cite{MyOwn6}. (This last result is also valid when not all diffusion coefficients $\si_n$ are equal to $1$.) However, a complete description of these weak limits for all initial distributions remains an unsolved problem.

\subsubsection{Gap process for two-sided infinite systems} In this paper, we explore the same questions as above for two-sided infinite systems $(X_n)_{n \in \MZ}$, for the case $\sup|g_n| < \infty$. We study stationary gap distributions,  as well as weak limits of $Z(t)$ as $t \to \infty$. Most of our results are for the case of unit diffusion coefficients: $\si_n = 1,\, n \in \MZ$; however, some of our results are for the general case. The results on weak limits are quite similar to the ones for one-sided infinite systems, with similar proofs. However, the results on stationary distributions are drastically different from both finite and one-sided infinite systems. We can have at least three possibilities:

\smallskip

(a) A family of product-of-exponentials stationary gap distributions $\pi_a$ indexed by one real parameter $a \in \BR$. An example of this is when all $g_n = 0$, or, more generally, when $\sum_{n \in \MZ}|g_n| < \infty$.

\smallskip

(b) A family of product-of-exponentials stationary gap distributions $\pi_{a, b}$ indexed by two real parameters $a, b \in \BR$. An example of this is when $g_n = 1,\, n > 0$; $g_n = 0,\, n \le 0$. 

\smallskip

(c) There are no stationary gap distributions, and $Z(t) \Ra \mathbf{0}$ as $t \to \infty$. An example of this is when $g_n = 1,\ n < 0$; $g_n = 0,\ n \ge 0$. 

\subsection{Motivation and historical review} These rank-based systems of competing Brownian particles were the subject of  extensive research in the last decade. Finite systems were studied in the following articles: \cite{IK2010, IKS2013, MyOwn3, MyOwn5, MyOwn14} (triple and multiple collisions of particles); \cite{PP2008, 5people}, \cite[Section 2]{MyOwn6} (stationary distribution $\pi$ for the gap process); \cite{JM2008, IPS2013, MyOwn10} (convergence $Z(t) \Ra \pi$ as $t \to \infty$ with an exponential rate); concentration of measure, \cite{Pal, PS2010}; see also miscellaneous papers \cite{JR2014, Reygner2015, MyOwn2, MyOwn10}. One-sided infinite systems of competing Brownian particles $(X_n)_{n \ge 1}$ were introduced in \cite{PP2008} and further studied in \cite{S2011, IKS2013, MyOwn6,  MyOwn13, DemboTsai}. 

\smallskip

Finite systems of competing Brownian particles have various applications: (a) financial mathematics, \cite[Chapter 5]{FernholzBook}, \cite{CP2010, FernholzKaratzasSurvey, MyOwn4, JR2013b}; (b) scaling limits of asymmetrically colliding random walks (a certain type of an exclusion process on $\MZ$), \cite[Section 3]{KPS2012}; (c) discretized version of a McKean-Vlasov equation, which governs {\it nonlinear diffusion processes}, and is related to the study of plasma, \cite{S2012, JR2013a, 4people, Reygner2015}. 

\smallskip

There are several generalizations of these models: (a) {\it systems with asymmetric collisions}, when ``particles have different mass'', studied in \cite{KPS2012} (finite systems) and \cite{MyOwn6} (one-sided infinite systems); (b) {\it second-order models}, when drift and diffusion coefficients depend on both ranks and names, \cite{5people, 2order}; (c) {\it systems of competing L\'evy particles}, with L\'evy processes instead of Brownian motions driving these particles, \cite{S2011, MyOwn12, MyOwn15}. 

\smallskip

Similar ranked systems of Brownian particles derived from independent driftless Brownian motions were studied  in \cite{Arratia1983, Harris, Sznitman1, Sznitman2}. The paper \cite{Harris} studied a two-sided infinite system of competing Brownian particles with zero drifts and unit diffusions: 
\begin{equation}
\label{eq:trivial}
g_n = 0,\ \mbox{and}\ \si_n = 1\ \mbox{for all}\  n.
\end{equation}
These particles $X_n,\, n \in \MZ$, can be alternatively described as independent Brownian motions. It was shown that if the initial distribution corresponds to a Poisson point process on the real line with constant intensity, then $\Var Y_0(t) \sim ct^{1/2}$ for an explicit constant $c$, as $t \to \infty$. More general results can be found in \cite[Theorem 3.7.1]{Book}, when particles in a two-sided infinite system can be fractional Brownian motions or more general processes. The paper \cite{Arratia1983} studied asymptotics for the lowest-ranked particle $Y_1$ in a one-sided infinite system of competing Brownian particles with  parameters as in~\eqref{eq:trivial}. See also the paper \cite{FSW2015} for totally asymmtetric collisions of driftless Brownian particles. 

\smallskip

Several other papers study connections between systems of queues and one-dimensional interacting particle systems: \cite{Ferrari1, Ferrari2, Kipnis, Timo}. Links to the GUE random matrix ensemble can be found in \cite{Barysh, OConnellYor}. Similar one-sided infinite systems of ranked particles in discrete time were studied in \cite{Aizenman1, Aizenman2}. In particular, in \cite{Aizenman1} they found stationary gap distributions for a discrete-time analogue of a one-sided infinite system with parameters~\eqref{eq:trivial}. See also related papers \cite{Bramson1, Bramson2, Tracer, Spitzer}. 

\smallskip

Let us also mention the paper \cite{Spohn} about relation between Dyson's Brownian motion and finite systems of competing Brownian particles with parameters as in~\eqref{eq:trivial}. The difference between Dyson's Brownian motion and systems of competing Brownian particles is that the logarithmic potential repels particles in the Dyson model, so that they cannot even hit each other. A recent paper \cite{Tsai} studies two-sided infinite systems of Dyson's Brownian particles. 


\subsection{Organization of the paper} Section 2 contains all our results about existence and  uniqueness of two-sided infinite systems, their basic properties, and approximation by finite systems. Section 3 is devoted to our results about the gap process: stationary gap distributions and long-term behavior of the gap process, for two-sided infinite systems. Section 4 contains all the proofs. The Appendix contains some technical lemmata and observations. 

\section{Existence, Uniqueness, and Basic Properties}

\subsection{Existence and uniqueness} We need some assumption on the initial condition $X(0) = x = (x_n)_{n \in \MZ} \in \BR^{\MZ}$; otherwise we cannot hope that even weak existence holds. Indeed, assume for simplicity that all $g_n = 0$, and all $\si_n = 1$. If $x_n = 0$ for every $n$, then $X_n,\, n \in \MZ$, are simply i.i.d. Brownian motions starting from zero. It is an easy exercise to show that the sequence $X(t) = (X_n(t))_{n \in \MZ}$ is not rankable for $t > 0$. Therefore, starting points $X_n(0) = x_n$ for each particle $X_n,\, n \in \MZ$, should be far enough apart. More precisely, they should be in the following subset of $\BR^{\MZ}$:
$$
\mathcal W := \Bigl\{x = (x_n)_{n \in \MZ} \in \BR^{\MZ}\,\, \Big{\vert}\, \SL_{n \in \MZ}e^{-\al x_n^2} < \infty\ \mbox{for all}\ \al > 0\Bigr\}.
$$
We say that a sequence $(a_n)_{n \in \MZ}$ of real numbers has {\it constant tails} if there exist $n_{\pm} \in \MZ$ such that $a_n = a_{n_+}$ for $n \ge n_+$, and $a_n = a_{n_-}$ for $n \le n_-$. 

\begin{thm} Assume $X(0) = x \in \mathcal W$ a.s., and at least one of the two following conditions holds:

\medskip

(a) $\si_n \equiv \si > 0$, $g_n \to g_{\infty}$ as $|n| \to \infty$, and $\sum_{n \in \MZ}(g_n - g_{\infty})^2 < \infty$;\, or

\medskip

(b) the sequences $(g_n)_{n \in \MZ}$ and $(\si_n)_{n \in \MZ}$ have constant tails.

\medskip

Then there exists in the weak sense a unique in law version of the two-sided infinite system of competing Brownian particles with drift  coefficients $(g_k)_{k \in \MZ}$ and diffusion coefficients $(\si_k^2)_{k \in \MZ}$, starting from $X(0) = x$. 
\label{thm:existence}
\end{thm}


\begin{rmk} Under assumptions of Theorem~\ref{thm:existence}, in both cases (a) and (b), we have:
\begin{equation}
\label{eq:bdd-seq}
\ol{g} := \sup\limits_{k \in \MZ}|g_k| < \infty,\ \mbox{and}\ \ol{\si} := \sup_{k \in \MZ}\si_k < \infty.
\end{equation}
\label{rmk:bdd-seq}
\end{rmk}

\subsection{Basic properties} The next statement represents a two-sided infinite system as a weak limit of finite systems, as the number of particles in these finite systems goes to infinity. Take a two-sided infinite system $X = (X_n)_{n \in \MZ}$ of competing Brownian particles with drifts $g_n$ and diffusions $\si_n^2$, $n \in \MZ$, starting from $X(0) = x = (x_n)_{n \in \MZ}$. Without loss of generality, assume the initial conditions are ranked: $x_n \le x_{n+1}$ for $n \in \MZ$. For every pair $M, N$ of integers such that $M < N$, consider a finite system of competing Brownian particles 
\begin{equation}
\label{eq:finite-X}
X^{(M, N)} = \left(X_M^{(M, N)}, \ldots, X_N^{(M, N)}\right)
\end{equation}
with drifts $g_M, \ldots, g_N$, and diffusions $\si_M^2, \ldots, \si_N^2$, starting from $(x_M, \ldots, x_N)$. Define the corresponding system of ranked particles:
\begin{equation}
\label{eq:finite-Y}
Y^{(M, N)} = \left(Y_M^{(M, N)}, \ldots, Y_N^{(M, N)}\right).
\end{equation}

\begin{defn}
A sequence $(M_j, N_j)_{j \ge 1}$ in $\MZ^2$ is called an {\it approximative sequence} if 
$$
M_{j+1} \le M_j < N_j \le N_{j+1}\ \mbox{for every}\ j \ge 1,
$$
$$
\lim\limits_{j \to \infty}M_j = -\infty,\ \lim\limits_{j \to \infty}N_j = \infty.
$$ 
Take any approximating sequence $(M_j, N_j)$. Then for every $k \in \MZ$, there exists a $j_k$ such that $M_j \le k < N_j$ for $j \ge j_k$. 
\label{defn:approx}
\end{defn}

\begin{lemma} Under assumptions of Theorem~\ref{thm:existence}, for every finite subset $I \subseteq \MZ$ and every $T > 0$, we have the following weak convergence in $C([0, T], \BR^{2|I|})$:
$$
\left([X^{(M, N)}]_I, [Y^{(M, N)}]_I\right) \Ra \left([X]_I, [Y]_I\right),\ \ (M, N) \to (-\infty, +\infty).
$$
That is, for every approximative sequence $(M_j, N_j)_{j \ge 1}$ from Definition~\ref{defn:approx}, every finite subset $I \subseteq \MZ$, and every $T > 0$, we have the following weak convergence in $C([0, T], \BR^{2|I|})$:
$$
\left([X^{(M_j, N_j)}]_I, [Y^{(M_j, N_j)}]_I\right) \Ra \left([X]_I, [Y]_I\right),\ \ j \to \infty.
$$
\label{lemma:approx}
\end{lemma}


We can extend the comparison techniques of \cite{MyOwn2, MyOwn6} for finite and one-sided infinite systems to two-sided infinite systems. Let us state one result, which is an analogue and a corollary of \cite[Corollary 3.11]{MyOwn6}
It is used later in this article.  

\begin{lemma}Take two copies, $X$ and $\ol{X}$, of a two-sided infinite system of competing Brownian particles, with the same drift and diffusion coefficients, but with different initial conditions, satisfying conditions of Theorem~\ref{thm:existence}. Let $Y$ and $\ol{Y}$ be the corresponding ranked versions, and let $Z$ and $\ol{Z}$ be the corresponding gap processes.

\smallskip

(a) If $Y(0) \preceq \ol{Y}(0)$, then $Y(t) \preceq \ol{Y}(t)$ for all $t \ge 0$.

\smallskip

(b) If $Z(0) \preceq \ol{Z}(0)$, then $Z(t) \preceq \ol{Z}(t)$ for all $t \ge 0$. 
\label{lemma:comp}
\end{lemma}

In the next lemma, we obtain the equation for the dynamics of ranked particles $Y_k,\, k \in \MZ$. Note that we do {\it not} impose assumptions of Theorem~\ref{thm:existence} here. 

\begin{lemma} Consider any two-sided infinite system of competing Brownian particles with drift coefficients $g_n$ and diffusion coefficients $\si_n^2$, starting from $X(0) = x \in \mathcal W$. 
Assume that
\begin{equation}
\label{eq:bdd-coeff}
\ol{g} := \sup\limits_{n \in \MZ}|g_n| < \infty,\ \mbox{and}\ \, \ol{\si} := \sup\limits_{n \in \MZ}\si_n < \infty.
\end{equation}
 
\smallskip

(a) Then for every interval $[u_-, u_+] \subseteq \BR$ and every $T > 0$, there exist a.s. only finitely many $n \in \MZ$ such that there exists a $t \in [0, T]$ for which we have: $X_n(t) \in [u_-, u_+]$. In other words, in a finite amount of time, every finite interval is visited by only finitely many particles. 

\smallskip

(b) The ranked particles $Y_k$, $k \in \MZ$, satisfy the following equations:
\begin{equation}
\label{eq:ranked}
Y_k(t) = Y_k(0) + g_kt + \si_kB_k(t) + \frac12L_{(k-1, k)}(t) - \frac12L_{(k, k+1)}(t),\ \ t \ge 0.
\end{equation}
Here, $B_k = (B_k(t),\, t \ge 0),\, k \in \MZ$, are i.i.d. Brownian motions, and for every $k \in \MZ,\, L_{(k, k+1)} = (L_{(k, k+1)}(t), t \ge 0)$ is the semimartingale local time process at zero of $Y_{k+1} - Y_k$.
\label{lemma:ranked}
\end{lemma}

\begin{rmk} Similar equations~\eqref{eq:ranked} hold for ranked particles in finite and one-sided infinite systems, with understanding that $L_{(0, 1)} \equiv 0$ for a one-sided infinite system $X = (X_n)_{n \ge 1}$, and similarly $L_{(0, 1)} \equiv 0$, $L_{(N, N+1)} \equiv 0$ for a finite system $X = (X_1, \ldots, X_N)$. 
\label{rmk:ranked-general}
\end{rmk}

An informal description of the dynamics of ranked particles from~\eqref{eq:ranked} is as follows. The $k$th ranked particle $Y_k$ moves as a Brownian motion with drift coefficient $g_k$ and diffusion coefficient $\si_k^2$, as long as it does not collide with adjacent ranked particles $Y_{k-1}$ and $Y_{k+1}$. When the particle $Y_k$ collides with $Y_{k+1}$, these two particles are pushed apart by an increase $\md L_{(k, k+1)}$ in the semimartingale local time $L_{(k, k+1)}$. This push $\md L_{(k, k+1)}$ is split evenly between these colliding particles: one-half $(1/2)\md L_{(k, k+1)}$ is added to $Y_{k+1}$ to push it up; and one-half $(1/2)\md L_{(k, k+1)}$ is subtracted from $Y_k$ to push it down. This way, the rankings $Y_k \le Y_{k+1}$ is preserved. Same principles apply to collision between particles $Y_k$ and $Y_{k-1}$. 

\smallskip

One can generalize this model by taking other nonnegative coefficients $q^+_{k+1}$ and $q^-_k$ instead of $1/2$. These coefficients should satisfy $q^+_{k+1} + q^-_k = 1$. This way, the share $q^+_{k+1}\md L_{(k, k+1)}$ is added to $Y_{k+1}$, and the share $q^-_{k}\md L_{(k, k+1)}$ is subtracted from $Y_k$. For finite and one-sided infinite systems, this was done respectively in \cite{KPS2012} and \cite{MyOwn6};  However, we shall not study this generalization in our paper.

\section{The Gap Process: Stationary Distributions and Weak Convergence} 

Define the mapping $\Phi : \BR^{\MZ}_+ \to \BR^{\MZ}$ as follows:
$$
\Phi_n(z) = 
\begin{cases}
z_0 + \ldots + z_{n-1},\ \ n \ge 1;\\
0,\ \ n = 0;\\
- z_{-1} - \ldots - z_{-n},\ \ n \le -1,
\end{cases}
\ \ n \in \MZ,\ \ \mbox{for}\ \ z = (z_n)_{n \in \MZ} \in \BR_+^{\MZ}. 
$$ 
This mapping has the following meaning in our context. Take $X = (X(t),\, t \ge 0)$,  a two-sided infinite system of competing Brownian particles. Let $Y = (Y(t),\, t \ge 0)$ be the corresponding system of ranked particles, and $Z = (Z(t),\, t \ge 0)$ be its gap process. Then 
$$
Y_n(t) = \Phi_n(Z(t)) + Y_0(t),\, t \ge 0,\, n \in \MZ.
$$
Define the following subset $\mathcal V \subseteq \BR_+^{\MZ}$: 
\begin{equation}
\label{eq:V-W}
\mathcal V := \{z = (z_k)_{k \in \MZ}\in \BR_+^{\MZ}\mid \Phi(z) \in \mathcal W\}.
\end{equation}
Then the following statements are equivalent: for every $t \ge 0$, 
$$
X(t) \in \mathcal W\ \Leftrightarrow\ Y(t) \in \mathcal W\ \Leftrightarrow\ Z(t) \in \mathcal V.
$$

\subsection{Stationary gap distributions for unit diffusions}  In this subsection, we assume 
\begin{equation}
\label{eq:unit-diffusions}
\si_n = 1\ \mbox{for all}\ n \in \MZ.
\end{equation}
The following theorem is the main result of this paper. 

\begin{thm} Assume conditions of Theorem~\ref{thm:existence} hold, together with~\eqref{eq:unit-diffusions}.  For any pair $(a, b) \in \BR^2$ of real numbers, consider the following sequence:
\begin{equation}
\label{eq:rates}
\la = (\la_n)_{n \in \MZ} \in \BR^{\MZ},\ \la_n = 2\Phi_{n+1}(g) + a + bn,\ n \in \MZ.
\end{equation}
If all $\la_n > 0$, then the following is a stationary gap distribution, supported on $\mathcal V$:
\begin{equation}
\label{eq:explicit-pi-a-b}
\pi _{a, b} := \bigotimes\limits_{n \in \MZ}\Exp(\la_n).
\end{equation}
\label{thm:stationary}.
\end{thm}

\begin{rmk} Define the following set:
$$
\Sigma := \{(a, b) \in \BR^2\,\mid\, \forall n \in \MZ,\ \la_n > 0\}.
$$
Take a probability measure $\rho$ on $\Sigma$. Then the following mixture of measures $\pi_{a, b},\, (a, b) \in \Sigma$:
$$
\int_{\Sigma}\pi_{a, b}\,\md\rho(a, b)
$$
is also a stationary gap distribution. 
\label{rmk:mixture}
\end{rmk}

Similarly to \cite[Conjecture 1.3]{MyOwn13} for one-sided systems, we can state a conjecture which is a converse to Remark~\ref{rmk:mixture}. 

\begin{conj} Under conditions of Theorem~\ref{thm:stationary}, every stationary gap distribution of the two-sided infinite system can be represented as in Remark~\ref{rmk:mixture}. 
\end{conj}


\begin{rmk} Every sequence $\la = (\la_n)_{n \in \MZ}$ from~\eqref{eq:rates} is a solution to the following difference equation:
\begin{equation}
\label{eq:difference}
\frac12\la_{n-1} - \la_n + \frac12\la_{n+1} = g_{n+1} - g_n,\ \ n \in \MZ.
\end{equation}
The converse is also true: Every solution to the difference equation~\eqref{eq:difference} has the form~\eqref{eq:rates} for some $a, b \in \BR$. 
\end{rmk}

\begin{exm} Let $g_n = 0$ for all $n$. In this case, $\Phi_n(g) = 0$ for all $n$. Therefore, $\la_n$ from~\eqref{eq:rates} satisfy $\la_n > 0$ for every $n \in \MZ$, if and only if $b = 0,\, a > 0$. This gives us $\la_n = a$ for all $n$. The stationary gap distributions have the form 
$$
\pi_a = \bigotimes_{n \in \MZ}\Exp(a),\ a > 0.
$$
This is actually a well-known result. Indeed, the gap distribution $\pi_a$ corresponds to the Poisson point process on the real line with intensity $a\,\md x$. But this Poisson point process is preserved under Brownian dymanics.
\end{exm}

\begin{exm} More generally, assume $\sum_{n \in \MZ}|g_n| < \infty$. Then the sequence $\Phi(g) = (\Phi_n(g))_{n \in \MZ}$ is bounded. Conditions of Theorem~\ref{thm:existence} (b) are satisfied, because $g_{\infty} = 0$, and $\sum|g_n| < \infty$ implies $\sum g_n^2 < \infty$. Similarly to Example 1, we have $\la_n > 0$ for all $n \in \MZ$, if and only if $b = 0,\, a > -\Phi_n(g)$ for all $n \in \MZ$. As in Example 1, we have a one-parameter family of stationary gap distributions.
\end{exm}

\begin{exm} Take the following drift coefficients:
$$
g_n = 
\begin{cases}
1,\ n \ge 1,\\ 
0,\ n \le 0.
\end{cases}
$$
Then $\Phi_{n+1}(g) = n\vee 0$, and $\la_n = a + bn + 2(n\vee 0)$ for $n \in \MZ$. We have: $\la_n > 0$ for $n \in \MZ$ if and only if $a > 0,\, b \in [-2, 0]$. In contrast with Examples 1 and 2, here we have a two-parameter family of stationary gap distributions. 
\end{exm}

\begin{exm} Take the following drift coefficients:
$$
g_n = 
\begin{cases}
1,\ n \le 0,\\
0,\ n \ge 1.
\end{cases}
$$
Then, similarly to Example 3, $\la_n = a + bn + 2(n\wedge 0)$ for $n \in \MZ$. There do not exist $a, b$ such that $\la_n > 0$ for all $n$. In other words, the set $\Sigma$ is empty: $\Sigma = \varnothing$. This is not accidental: In fact, as we shall see later, this system does not have any stationary gap distributions at all: regardless of the initial conditions, $Z(t)$ weakly converges to zero as $t \to \infty$. 
\end{exm}

\subsection{Long-term behavior of the gap process for general diffusions} In this subsection, we do {\it not} assume~\eqref{eq:unit-diffusions}. For $M < N$, define the following quantity:
$$
\ol{g}[M:N] := \frac1{N-M+1}\left(g_M + \ldots + g_N\right).
$$ 

\begin{asmp} There exists an approximative sequence $(M_j, N_j)_{j \ge 1}$ such that 
$$
\ol{g}[M_j:k] > \ol{g}[M_j:N_j],\ \ k = M_j, \ldots, N_j - 1,\ \ j \ge 1.
$$
\end{asmp}

Consider a system $X^{(M_j, N_j)}$ as in~\eqref{eq:finite-X}, but without a given initial condition. It follows from~\eqref{eq:stability-intro} that, under Assumption 1, the system $X^{(M_j, N_j)}$ has a unique stationary gap distribution. Denote this distribution by $\pi^{(j)}$; this is a probability measure on $\BR_+^{N_j - M_j}$. 

\begin{lemma}
Take a $j \ge 1$ and a subset $I \subseteq \{M_j, \ldots, N_j - 1\}$. For $j' > j$, the marginals of stationary gap distributions $\pi^{(j)},\, \pi^{(j')}$ satisfy 
$$
\bigl[\pi^{(j')}\bigr]_I  \preceq \bigl[\pi^{(j)}\bigr]_I.
$$
\label{lemma:comparison-of-stationary-gap-distributions}
\end{lemma}

Therefore, we can couple all these stationary distributions: take random variables 
$$
(z^{(j)}_{M_j}, \ldots, z^{(j)}_{N_j - 1}) \sim \pi^{(j)},\ \ j \ge 1,
$$
so that the following comparison holds a.s.:
$$
z_k^{(j)} \ge z_k^{(j+1)}\ \mbox{for}\ M_j \le k < N_j,\, j \ge 1.
$$
For every $k \in \MZ$, define the limits 
$$
z_k^{(\infty)} := \lim\limits_{j \to \infty}z^{(j)}_k \in \BR_+.
$$
Denote by $\pi^{(\infty)}$ the distribution of the random vector $(z^{(\infty)}_k)_{k \in \MZ}$ in $\BR_+^{\MZ}$. 

\begin{rmk} We also note that this limiting distribution is independent of the approximative sequence $(M_j, N_j)$: If we take two different approximative sequences $(M_j, N_j)$ and $(\tilde{M}_j, \tilde{N}_j)$, each satisfying Assumption 1, then the resulting limiting distributions $\pi^{(\infty)}$ and $\tilde{\pi}^{(\infty)}$ are the same: $\pi^{(\infty)} = \tilde{\pi}^{(\infty)}$. The proof is similar to that of \cite[Lemma 4.2]{MyOwn6} and is omitted. 
\end{rmk}

Take any copy $X = (X_n)_{n \in \MZ}$ of a two-sided infinite system of competing Brownian particles with drift coefficients $g_n,\, n \in \MZ$, and diffusion coefficients $\si_n^2,\, n \in \MZ$, starting from any initial conditions. Let $Z = (Z(t),\, t \ge 0)$ be the gap process.

\begin{thm}
\label{thm:no-unit-diffusions}
Under Assumption 1, 

\medskip

(a) the family of $\BR_+^{\MZ}$-valued random variables $Z = (Z(t), t \ge 0)$ is tight;

\medskip

(b) all weak limit points of $Z(t)$ as $t \to \infty$ and all stationary gap distributions are stochastically dominated by the measure $\pi^{(\infty)}$.
\end{thm}

In Theorem~\ref{thm:no-unit-diffusions}, we do {\it not} impose assumptions of Theorem~\ref{thm:existence}. If we do impose them, we can get some additional results.

\begin{thm}
\label{thm:conv-general}
Under Assumption 1 and conditions of Theorem~\ref{thm:existence}, 

\medskip

(a) if $\pi^{(\infty)}$ is supported on $\mathcal V$, then $\pi^{(\infty)}$ is a stationary gap distribution;

\medskip

(b) if, in addition, $\pi^{(\infty)} \preceq Z(0)$, then $Z(t) \Ra \pi^{(\infty)}$ as $t \to \infty$. 
\end{thm}

If~\eqref{eq:unit-diffusions} does not hold, then generally we do not know an explicit formula for $\pi^{(j)}$ and $\pi^{(\infty)}$. However, under condition~\eqref{eq:unit-diffusions}, we can show more explicit results. The next subsection is devoted to this. 

\subsection{Long-term behavior of the gap process for unit diffusions} In this subsection, we assume~\eqref{eq:unit-diffusions}. This allows us to get more explicit results than in the previous subsection. Without loss of generality, assume $j_0 = 1$. Let
$$
\la_k^{(j)} := 2(k - M_j + 1)\left(\ol{g}[M_j:k] - \ol{g}[M_j:N_j]\right),\ \ M_j \le k < N_j,\ j \ge 1.
$$ 
Under Assumption 1, these quantities are all positive: 
$$
\la_k^{(j)} > 0,\ M_j \le k < N_j,\ j \ge 1.
$$
Under Assumption 1 and~\eqref{eq:unit-diffusions}, the formula~\eqref{eq:product-form} gives us
$$
\pi^{(j)} = \bigotimes\limits_{k = M_j}^{N_j - 1}\Exp\bigl(\la_k^{(j)}\bigr),\, j \ge 1.
$$

\begin{lemma} Each sequence $(\la^{(j)}_k)_{j \ge j_k}$ is nondecreasing. Consider the limits
\begin{equation}
\label{eq:limits}
\la_k^{(\infty)} := \lim\limits_{j \to \infty}\la^{(j)}_k \in (0, \infty],\ \ k \in \MZ.
\end{equation}
Then either all $\la_k^{(\infty)},\, k \in \MZ$, are finite, or all are infinite. 
\label{lemma:rate-limits}
\end{lemma}

Therefore (understanding $\Exp(\infty) = \de_0$ to be the Dirac mass at zero), we get:
$$
\pi^{(\infty)} = \bigotimes\limits_{k \in \MZ}\Exp\bigl(\la_k^{(\infty)}\bigr).
$$

Depending on whether all $\la_k^{(\infty)}$ are finite or infinite, we get a different long-term behavior of $Z(t)$. The following result is a corollary of Theorems~\ref{thm:no-unit-diffusions} and~\ref{thm:conv-general}. 

\begin{thm} Under Assumption 1, condition~\eqref{eq:unit-diffusions}, and assumptions   of Theorem~\ref{thm:existence}, suppose $\la_k^{(\infty)} < \infty,\, k \in \MZ$. Then:

\medskip

(a) the sequence $(\la_k^{(\infty)})_{k \in \MZ}$ satisfies assumptions of Theorem~\ref{thm:stationary}, and therefore 
$$
\pi^{(\infty)} = \bigotimes\limits_{k \in \MZ}\Exp\bigl(\la^{(\infty)}_k\bigr)
$$
is a stationary gap distribution.

\medskip

(b) any weak limit point of $Z(t)$ as $t \to \infty$, as well as any other stationary gap distribution, is stochastically dominated by $\pi^{(\infty)}$;

\medskip

(c) if $\pi^{(\infty)} \preceq Z(0)$, then $Z(t) \Ra \pi^{(\infty)}$ as $t \to \infty$.
\label{thm:conv}
\end{thm}

Theorem~\ref{thm:conv} (c) provides a partial description of the domain of convergence for the stationary gap distribution $\pi^{(\infty)}$; that is, for which initial distributions $Z(0)$ we have $Z(t) \Ra \pi^{(\infty)}$. To the best of our knowledge, it is still an unsolved problem to completely describe this domain of convergence, as well as domains of convergence for other stationary gap distributions $\pi$.

\begin{exm} Take the following drift coefficients: $g_n > 0,\, n \le 0;\, g_n = 0,\, n \ge 1$, with 
\begin{equation}
\label{eq:finite-sum}
\sum\limits_{n \le 0}g_n < \infty.
\end{equation}
Try $M_j = -j+1$ for $j \ge 1$. Then 
$$
\ol{g}[M_j:k] = \frac{1}{k + j}\SL_{l = -j+1}^{k\wedge 0}g_l,\ \ k > M_j.
$$
We can find an $N_j$ large enough so that 
\begin{equation}
\label{eq:lambda-651}
\la_k^{(j)} := 2\SL_{n = -j+1}^{k\wedge 0}g_n - 2\frac{k+j}{N_j+j}\SL_{n = -j+1}^{0}g_n > 0,\ \ k = -j+1, \ldots, 0.
\end{equation}
If we take $N_j > j^2$, then $(k+j)/(N_j+j) \to 0$ as $j \to \infty$. Therefore, from~\eqref{eq:finite-sum} and~\eqref{eq:lambda-651}, we get:
$$
\la^{(\infty)}_k := \lim\limits_{j \to \infty}\la^{(j)}_k = 2\sum_{n < k\wedge 0}g_n < \infty,\ \ k \in \MZ. 
$$
Thus, we can apply Theorem~\ref{thm:conv}. 
\end{exm}

\begin{thm} Under Assumption 1 and condition~\eqref{eq:unit-diffusions}, suppose all $\la^{(\infty)}_k = \infty,\, k \in \MZ$. Then, regardless of initial conditions, $Z(t) \Ra \mathbf{0}$ as $t \to \infty$. 
\label{thm:conv-0}
\end{thm}

\begin{exm} In Example 4 above, let $M_j = -j+1,\, N_j = j,\, j \ge 1$. From Theorem~\ref{thm:conv-0}, we get:
$$
\la_k^{(j)} = 2(j + k\wedge0),\ \ \mbox{and}\ \ \la_k^{(\infty)} = \lim_{j \to \infty}\la_k^{(j)} = \infty.
$$
\end{exm}


\section{Proofs} 

\subsection{Proof of Theorem~\ref{thm:existence}} {\it Proof of (a).} This  is similar to that of \cite[Theorem 3.2]{MyOwn6} and is based on Girsanov change of measure. We shall not repeat it here in full detail. However, noting that we start the construction from a system $X = (X_i)_{i \in \MZ}$ of independent Brownian motions starting from $X_i(0) = x_i,\, i \in \MZ$, we shall prove the following fact: 

\begin{lemma} For every $t \ge 0$, the system $X(t) = (X_i(t))_{i \in \MZ}$ is rankable, and one can choose ranking permutations $\mP_t,\, t \ge 0$, which satisfy the property (b) of Definition~\ref{defn:CBP}. 
\label{lemma:right-ranking}
\end{lemma}

\begin{proof} Take an interval $[u_-, u_+] \subseteq \BR$ and a time horizon $T > 0$. From \cite[Lemma 7.1, Lemma 7.2]{MyOwn6}, the Borel-Cantelli lemma, and the fact that $X(0) = x \in \mathcal W$ a.s., it follows that a.s. there exist only finitely many $n \ge 1$ such that $\min_{t \in [0, T]}X_n(t) > u_+$, and only finitely many $n \le -1$ such that $\max_{t \in [0, T]}X_n(t) < u_-$. Therefore, there exist a.s. only finitely many $n \in \MZ$ such that $\exists,\ t \in [0, T]:\, X_n(t) \in [u_-, u_+]$. In particular, for every $t \ge 0$, we have: 
\begin{equation}
\label{eq:limits-infty}
\lim\limits_{n \to \infty}X_n(t) = \infty,\ \ \lim\limits_{n \to -\infty}X_n(t) = -\infty.
\end{equation}
Any two Brownian motions collide on a set of times which a.s. has Lebesgue measure zero. The union of countably many zero probability events is itself a zero probability event; therefore, a.s.
\begin{equation}
\label{eq:no-collisions}
\mes\{t \ge 0\mid \exists\, m, n \in \MZ,\, m \ne n:\ X_m(t) = X_n(t)\} = 0.
\end{equation}
Now apply \cite[Theorem 3.1]{Harris} and complete the proof. 
\end{proof}

\smallskip

{\it Proof of (b).} It is quite similar to the one for one-sided infinite systems, given in \cite{S2011, IKS2013}, and \cite[Theorem 3.1]{MyOwn6}. However, there are some differences, so we present the full proof here.  Without loss of generality, assume $x_n \le x_{n+1}$ for $n \in \MZ$. By assumptions, the sequences $(g_n)_{n \in \MZ}$  and $(\si_n)_{n \in \MZ}$ have constant tails. Therefore, there exist some $n_{\pm} \in \MZ,\, g_{\pm} \in \BR,\, \si_{\pm} > 0$, such that
\begin{equation}
\label{eq:constant-tails}
g_n = g_+,\ \si_n = \si_{+},\ n \ge n_+;\ \ g_n = g_{-},\ \si_n = \si_{-},\ n \le n_-.
\end{equation}

\subsubsection{The idea of the construction} In the beginning, we have particles with ranks $n_-+1, \ldots, n_+-1$, which behave in a complicated way (as competing Brownian particles), and other particles, which behave simply as independent Brownian motions. We construct the two-sided infinite system as consisting of three parts: particles with ranks $n_-+1, \ldots, n_+-1$, which form a finite system of competing Brownian particles; infinitely many particles with ranks $n_+, n_++1, \ldots$, which behave as Brownian motions with drift coefficients $g_+$ and diffusion coefficients $\si_+^2$; and infinitely many particles with ranks $\ldots, n_--1, n_-$, which behave as Brownian motions with drift coefficients $g_-$ and diffusion coefficients $\si_-^2$. As long as a particle $X_n$ from the second or third part does not hit particles with ranks $n_-+1, \ldots, n_+-1$, this particle $X_n$ continues to behave as a Brownian motion. If this particle $X_n$ hits a particle from the first part at a certain time $\tau_1$, we remove $X_n$ from the second or third part, and add it to the first part. We do this for all particles from the second or third part which hit a particle from the first part at this moment $\tau_1$. Then we run this system again, until the next such hitting time $\tau_2$. The first part of this system increases at every time $\tau_m$. 

\subsubsection{Formal construction} For every pair $(M, N)$ of integers such that $M \le N$, and for every $x \in \BR^{N-M+1}$, take a probability space $\left(\Oa^{(M, N, x)}, \CF^{(M, N, x)}, \MP^{(M, N, x)}\right)$ with a system of $N - M + 1$ competing Brownian particles:
$$
X^{(M, N, x)} = \left(X^{(M, N, x)}_{M}, \ldots, X^{(M, N, x)}_N\right)
$$
with drift coefficients $(g_n)_{M \le n \le N}$ and diffusion coefficients $(\si_n^2)_{M \le n \le N}$, starting from $X^{(M, N, x)}(0) = (x_{M}, \ldots, x_N)$. Take yet another probability space with i.i.d. Brownian motions $W^{(j)}_k,\ j \ge 0,\ k \in \MZ$. Now, consider the product $(\Oa, \CF, \MP)$ of all these probability spaces. Define the infinite system $X$ by induction: We simultaneously construct an increasing sequence of stopping times $(\tau_m)_{m \ge 0}$, and the system $X$ on each time interval $[\tau_m, \tau_{m+1}]$, for each $m \ge 0$. First, we define 
$$
I_0 := \{n_-+1, \ldots, n_+-1\},\ \ \tau_0 := 0, 
$$
$$
J_0^+ := \{n_+, n_+ + 1, \ldots\},\ J_0^- := \{\ldots, n_--1, n_-\},
$$
and construct the system of particles: for $t \le \tau_1$, 
$$
X_k(t) := 
\begin{cases}
X_k^{(n_-, n_+)}(t),\ \ k \in I_0;\\
x_k + g_+t + \si_+W^{(0)}_k(t),\ \ k \ge n_+;\\
x_k + g_-t + \si_-W_k^{(0)}(t),\ \ k \le n_-.
\end{cases}
$$
Next, we define by induction
$$
\tau_{m+1} := \inf\{t \ge \tau_m\mid \exists\, i \in \MZ\setminus I_m,\ j \in I_0:\ X_i(t) = Y_j(t)\},
$$
$$
I_{m+1} := I_m\cup\{i \in \MZ\mid \exists\, j \in I_0: X_i(\tau_{m+1}) = Y_j(\tau_{m+1})\}.
$$
For each $m = 0, 1, \ldots$ let $J^+_m := J^+_0\setminus I_m$, $J^-_m := J^-_0\setminus I_m$. Let $n_-(m)$ and $n_+(m)$ be the minimal and maximal ranks of particles $X_i(\tau_m)$ at time $\tau_m$ with names $i$ in $I_m$. It is easy to prove by induction that the set of ranks of particles $X_i(\tau_m)$ with $i \in I_m$ is exactly $\{n_-(m), \ldots, n_+(m)\}$. Next, for every $m = 0, 1, \ldots$ and for $t \le \tau_{m+1} - \tau_m$, we define: $x_m := (X_{i}(\tau_m))_{i \in I_m}$, and 
$$
X_i\left(t+\tau_m\right) = 
\begin{cases}
X_i(\tau_m) + g_+t + \si_+W^{(m)}_i(t),\ \ i \in J^+_m;\\
X_i(\tau_m) + g_-t + \si_-W^{(m)}_i(t),\ \ i \in J^-_m;\\
X^{(x_m,\, n_-(m),\, n_+(m))}_i(t),\ i \in I_m. 
\end{cases}
$$

Assume we proved the following statements. 

\begin{lemma} 
\label{lemma:rankable}
For every $m = 1, 2, \ldots$ and every $t < \tau_m$, the vector $X(t) = (X_i(t))_{i \in \MZ}$ is rankable. 
\end{lemma}

\begin{lemma}
\label{lemma:finite-sets}
For every $m = 1, 2, \ldots$ a.s. the set $I_m$ is finite. 
\end{lemma}

\begin{lemma}
\label{lemma:correct-choice}
For every $m = 1, 2, \ldots$ and $t \in [0, \tau_m]$, there exists a ranking permutation $\mP_t$ of $X(t)$ so that condition (b) from Definition~\ref{defn:CBP} is satisfied on $[0, \tau_m]$. 
\end{lemma}

\begin{lemma}
\label{lemma:infinite-limit}
As $m \to \infty$, we have: $\tau_m \to \infty$ a.s. 
\end{lemma}

Using induction by $m$, together with Lemmata~\ref{lemma:rankable},~\ref{lemma:finite-sets},~\ref{lemma:correct-choice}, we get that until $\tau_m$, this is a system with required properties. By Lemma~\ref{lemma:infinite-limit}, this statement is true on the infinite time horizon. Uniqueness in law can be also proved in a straightforward way on using induction by $m$. This has been done in \cite{S2011, IKS2013}, and we shall not repeat all details here.

\subsubsection{Proof of Lemma~\ref{lemma:rankable}} Fix time horizon $T > 0$. Let us prove this statement for $\tau_m\wedge T$ instead of $\tau_m$. We use induction by $m$. For $m = 0$, there is nothing to prove. If $I_{m-1}$ is finite, it suffices to show that, for a given level $u \in \BR$, during the time interval $[0, \tau_m\wedge T]$,

\smallskip

(a) $\min_{t \in [0, \tau_m\wedge T]}X_i(t) \le u$ for only finitely many particles $X_i(t),\, i \in J^+_{m-1}$, a.s.

\smallskip

(b) $\max_{t \in [0, \tau_m\wedge T]}X_i(t) \ge u$ for only finitely many particles $X_i(t),\, i \in J^-_{m-1}$, a.s.

\smallskip

Particles from (a) and (b) are Brownian motions with drift and diffusion $g_+,\, \si_+^2$ and $g_-,\, \si_-^2$, respectively. Apply \cite[Lemmata 7.1, 7.2]{MyOwn6} together with the Borel-Cantelli lemma, and complete the proof of (a) and (b), together with Lemma~\ref{lemma:rankable}. 

\subsubsection{Proof of Lemma~\ref{lemma:finite-sets}} Assume the converse, and denote this event by $A_{\infty}$. If this event happened, then for some $m$, the set $I_{m-1}$ is finite, but the set $I_m$ is infinite. Therefore, we can represent 
$A_{\infty}$ as
\begin{equation}
\label{eq:countable-union}
A_{\infty} = \bigcup\limits_{J}^{\infty}A(m, J),\ \ 
\mbox{where}\ \ A(m, J) := \{I_{m-1} = J,\ I_m\ \mbox{is infinite}\}.
\end{equation}
Here, the union is taken over all finite sets $J \subseteq \MZ$. This union is countable. 
Assume the event $A(m, J)$ has happened. Then $\tau_{m} < \infty$. The fact that $I_m$ is infinite means that $X_i\left(\tau_{m}\right)$ is the same for infinitely many values of $i \in \MZ\setminus J$. But even three (let alone infinitely many) independent Brownian motions can collide only with probability zero. That is, if $W_1, W_2, W_3$ are independent one-dimensional Brownian motions, then 
$$
\MP\left(\exists\, t > 0:\ W_1(t) = W_2(t) = W_3(t)\right) = 0.
$$
Therefore, $\MP(A(m, J)) = 0$. Thus, from~\eqref{eq:countable-union} we have: $\MP(A_{\infty}) = 0$.

\subsubsection{Proof of Lemma~\ref{lemma:correct-choice}} Similar to the proof of Lemma~\ref{lemma:correct-choice}: we need to apply \cite[Theorem 3.1]{Harris}. The property~\eqref{eq:limits-infty} follows from properties (a) and (b) in the proof of Lemma~\ref{lemma:rankable}. The property~\eqref{eq:no-collisions} holds for $t \in [\tau_m, \tau_{m+1}]$ because: (a) for each of the three parts of the system, we can prove it separately; (b) by construction, on the time interval $(\tau_m, \tau_{m+1})$, particles from different parts of the system (for example, from the first and the second part) do not collide. Similarly, we can show the property (b) from Definition~\ref{defn:CBP} by considering each of the three parts separately. 

\subsubsection{Proof of Lemma~\ref{lemma:infinite-limit}} Fix time horizon $T > 0$. Assume we proved that 
\begin{equation}
\label{eq:exceed}
\forall\, \eps > 0\ \exists\, u_{\eps}:\ \forall\, m\ \ \MP\left(\max\limits_{t \le \tau_m\wedge T}Y_{n_+-1}(t) > u_{\eps}\right) < \eps.
\end{equation}
The event $A = \{\lim_{m \to \infty}\tau_m \le T\}$ means infinitely many particles $X_i$ hit at least one of ranked particles $Y_{k},\ k \in I_0$, during the time interval $[0, T]$. Without loss of generality, assume there are infinitely many $i \ge n_+$ such that this holds. Until each of these hits, $X_i$ behaves as a Brownian motion with drift and diffusion coefficients $g_+,\, \si_+^2$. Note that $X_i(0) \ge Y_{n_+-1}(0)$ for $i \ge n_+$. Because they have continuous trajectories, these particles $X_i$ hit the ranked particle $Y_{n_+-1}$ first among these ranked particles $Y_k,\, k \in I_0$. Denote 
$$
B(\eps) := \left\{\exists\, m \ge 0:\ \max\limits_{t \le \tau_m\wedge T}Y_{n_+-1}(t) > u_{\eps}\right\}.
$$
Assume the event $A\setminus B(\eps)$ has happened. A particle $X_i$ hit a particle $Y_{n_+-1}$ at some time $t \in [0, T]$, when the particle $Y_{n_+-1}$ was below the level $u_{\eps}$. This particle $X_i$ has continuous trajectories, and therefore it hit the level $u_{\eps}$ at some time $t \in [0, T]$. Moreover, there are infinitely many such particles $X_i$. In other words, if the event $A\setminus B(\eps)$ has happened, then infinitely many Brownian motions, starting from $x_i$, $i \ge n_+$, hit level $u_{\eps}$ during the time interval $[0, T]$. Because $x \in \mathcal W$, we have:
$$
\SL_{i=n_+}^{\infty}e^{-\al x_i^2} < \infty\ \mbox{for all}\ \al > 0.
$$
Applying \cite[Lemmata 7.1, 7.2]{MyOwn6}, and the Borel-Cantelli lemma, we get: $\MP(A\setminus B(\eps)) = 0$. But from~\eqref{eq:exceed} we get: $\MP(B(\eps)) < \eps$. Therefore, 
$$
\MP(A) \le \MP(A\setminus B(\eps)) + \MP(B(\eps)) < \eps.
$$
Since $\eps > 0$ is arbitrary, we conclude that $\MP(A) = 0$. 

\smallskip

Now, let us show~\eqref{eq:exceed}. Consider a (one-sided) infinite system of competing Brownian particles $\ol{X} = (\ol{X}_i)_{i < n_+}$ with drifts $g_n,\, n < n_+$, diffusions $\si_n^2,\, n < n_+$, starting from $\ol{X}_i(0) = x_i$. (This system is inverted: it has the top-ranked particle but not the bottom-ranked particle. It is straightforward to adjust definitions, existence and uniqueness results, and comparison techniques from \cite{MyOwn6} for this case.) From~\eqref{eq:constant-tails}, we have: $g_n = g_{-}$ and $\si_n = \si_{-}$ for $n \le n_-$. Next,  $x \in \mathcal W$, and therefore
$$
\SL_{n < n_+}e^{-\al x_n^2} < \infty\ \mbox{for all}\ \al > 0.
$$
From \cite[Theorem 3.1]{MyOwn6} (suitably adjusted for the inverted one-sided infinite system), there exists a unique in law weak version of this system $\ol{X}$. Next, fix an $m$. By construction of the system, until $\tau_m$, the particle $Y_{n_+-1}$ behaves as a ranked particle in the finite system $Y^{(n_-(m), n_+(m))}$. The one-sided infinite system $\ol{X}$ can be obtained from this finite system by removing the top $n_+(m) - n_++1$ ranked particles from the top, and adding infinitely many ranked particles to the bottom. It follows from comparison techniques, similar to \cite[Corollary 3.11]{MyOwn6}, that we can couple these two ranked systems so that $Y_{n_+-1}(t) \le \ol{Y}_{n_+-1}(t)$.  It suffices to find $u_{\eps}$ large enough so  that 
$$
\MP\left(\max\limits_{0 \le t \le T}\ol{Y}_{n_+-1}(t) > u_{\eps}\right) < \eps.
$$

\subsection{Proof of Lemma~\ref{lemma:approx}} This proof is somewhat lengthy, and we split it into a few lemmata. In the first subsection, we enunciate them and show how they combine to form the whole proof. In later subsections, we prove these lemmata.

\subsubsection{Overview of the proof of Lemma~\ref{lemma:approx}}

We follow the proof of \cite[Theorem 3.3]{MyOwn6}, with minor adjustments. Consider an approximating sequence $(M_j, N_j)$. 

\begin{lemma} For every $i \in \MZ$, the sequence $\bigl(X^{(M_j, N_j)}_i\bigr)_{j \ge 1}$ is tight in $C[0, T]$.
\label{lemma:tightness-of-X}
\end{lemma}

\begin{lemma} For every $k \in \MZ$, the sequence $\bigl(Y^{(M_j, N_j)}_k\bigr)_{j \ge 1}$ is tight in $C[0, T]$.
\label{lemma:tightness-of-Y}
\end{lemma}

The proofs of Lemmata~\ref{lemma:tightness-of-X} and~\ref{lemma:tightness-of-Y} are given later in this subsection. Assuming we already proved these lemmata, let us finish the proof of Lemma~\ref{lemma:approx}. 

For every $j \ge 1$, let $W^{(M_j, N_j)} = (W_i^{(M_j, N_j)})_{M_j \le i \le N_j}$ be the sequence of driving Brownian motions for the system $X^{(M_j, N_j)}$ of competing Brownian particles. Then for every finite subset $I \subseteq \MZ$, we can extract a subsequence $(M'_j, N'_j)_{j \ge 1}$ of $(M_j, N_j)_{j \ge 1}$ such that there exist continuous adapted $\BR^{|I|}$-valued processes 
$$
X_I = (X_i)_{i \in I},\ X_i = (X_i(t),\, 0 \le t \le T),\ i \in I,
$$
$$
Y_I = (Y_i)_{i \in I},\ Y_i = (Y_i(t),\, 0 \le t \le T),\ i \in I,
$$
$$
W_I = (W_i)_{i \in I},\ W_i = (W_i(t), 0 \le t \le T),\ i \in I,
$$
for which we have the following convergence in $C([0, T], \BR^{3|I|})$, 
\begin{equation}
\label{eq:converged}
\Bigl(\bigl[X^{(M'_j, N'_j)}\Bigr]_I, \Bigl[Y^{(M'_j, N'_j)}\Bigr]_I, \Bigl[W^{(M'_j, N'_j)}\Bigr]_I\Bigr) \Ra (X_I, Y_I, W_I).
\end{equation}
Using the standard diagonal arguments, we can find a subsequence $(M'_j, N'_j)_{j \ge 1}$ which is {\it independent} of $I$. Then there exist $\BR^{\MZ}$-valued continuous processes 
$$
X = (X_i)_{i \in \MZ},\ \ X_i = (X_i(t),\, 0 \le t \le T),\ i \in \MZ,
$$
$$
Y = (Y_i)_{i \in \MZ},\ \ Y_i = (Y_i(t),\, 0 \le t \le T),\ i \in \MZ,
$$
$$
W = (W_i)_{i \in \MZ},\ \ W_i = (W_i(t),\, 0 \le t \le T),\ i \in \MZ,
$$
such that we have the following equality in law:
$$
\left([X(t)]_I, 0 \le t \le T\right) = (X_I(t),\, 0 \le t \le T),
$$
$$
\left([Y(t)]_I, 0 \le t \le T\right) = (Y_I(t),\, 0 \le t \le T),
$$
$$
\left([W(t)]_I, 0 \le t \le T\right) = (W_I(t),\, 0 \le t \le T),
$$
In fact, $W_i,\ i \in \MZ$, are i.i.d. Brownian motions, because these are weak limits of i.i.d. Brownian motions in~\eqref{eq:converged}. By the Skorohod representation theorem, we can assume a.s. convergence instead of the weak one (possibly after changing the probability space). By construction, the following sets of points are equal for all $t \in [0, T]$:
$$
\{X_i(t)\mid i \in \MZ\} = \{Y_k(t)\mid k \in \MZ\}.
$$

\medskip

\begin{lemma} 
\label{lemma:no-ties}
For every $t \in [0, T]$, a.s. there is no tie in the vector $Y(t) = (Y_k(t))_{k \in \MZ}$. 
\end{lemma}

Lemma~\ref{lemma:no-ties} can be equivalently stated as follows:  the set $\{t \in [0, T]\mid \exists\, k \ne l:\, Y_k(t) = Y_l(t)\}$ has Lebesgue measure zero. Its proof is postponed until the end of this subsection. The rest of the proof of Lemma~\ref{lemma:approx} closely follows that of \cite[Theorem 3.3]{MyOwn6}, and we do not repeat it here. 

\subsubsection{Proof of Lemma~\ref{lemma:tightness-of-X}} For $j \ge 1$, define
\begin{equation}
\label{eq:drift-j}
\be^{(M_j, N_j)}_i(s) := \SL_{k=M_j}^{N_j}1(X_i^{(M_j, N_j)}(s)\ \mbox{has rank}\ k)\,g_k,
\end{equation}
\begin{equation}
\label{eq:diffusion-j}
\rho^{(M_j, N_j)}_i(s) := \SL_{k=M_j}^{N_j}1(X_i^{(M_j, N_j)}(s)\ \mbox{has rank}\ k)\,\si_k.
\end{equation}
We can represent $X_i^{(M_j, N_j)}$ for $t \ge 0,\ M_j \le i \le N_j$, as
\begin{equation}
\label{eq:X-repr}
X^{(M_j, N_j)}_i(t) = x_i + \int_0^t\be^{(M_j, N_j)}_i(s)\,\md s + \int_0^t\rho^{(M_j, N_j)}_i(s)\,\md W_i^{(M_j, N_j)}(s),
\end{equation}
where $W^{(M_j, N_j)}_i$, $i = M_j, \ldots, N_j$, are i.i.d. Brownian motions. From~\eqref{eq:bdd-seq},~\eqref{eq:drift-j},~\eqref{eq:diffusion-j}, we get:
\begin{equation}
\label{eq:bounded-coefficients}
\bigl|\be^{(M_j, N_j)}_i(s)\bigr| \le \ol{g},\ \ \bigl|\rho^{(M_j, N_j)}_i(s)\bigr| \le \ol{\si}.
\end{equation}
It suffices to apply \cite[Lemma 7.4]{MyOwn6} and finish the proof. 

\subsubsection{Proof of Lemma~\ref{lemma:tightness-of-Y}} Fix a $k \in \MZ$. For all $j$, we have: $Y_k^{(j)}(0) = x_k$. Without loss of generality, we can shift this system and assume $Y_k^{(j)}(0) = 0$ for all $j \ge j_k$. 

\begin{lemma}
\label{lemma:comp-0}
For every $\eta > 0$, there exist $u_{\pm}$ such that for every $j \ge j_k$, we have:
\begin{equation}
\label{eq:comp-0}
\MP\left(\forall \ t \in [0, T],\ \ u_- \le Y^{(M_j, N_j)}_k(t) \le u_+\right) \ge 1 - \eta.
\end{equation}
\end{lemma}

\begin{proof} 
Take a one-sided infinite system $\ol{X} = (\ol{X}_n)_{n \ge k}$ of competing Brownian particles with drift coefficients $(g_n)_{n \ge k}$, diffusion coefficients $(\si_n^2)_{n \ge k}$, starting from $\ol{X}_n(0) = x_n,\, n \ge k$. From $x \in \mathcal W$, we have:
\begin{equation}
\label{eq:series-122}
\SL_{n=k}^{\infty}e^{-\al x_n^2} < \infty\ \mbox{for all}\ \al > 0.
\end{equation}
Using~\eqref{eq:constant-tails} and~\eqref{eq:series-122}, and applying \cite[Theorem 3.1]{MyOwn6}, we get: This system $\ol{X}$ exists in the weak sense and is unique in law. Denote by $\ol{Y} = (Y_k, Y_{k+1}, \ldots)$ the corresponding system of ranked particles. 
One can get the system $\ol{X}$ from $X^{(M_j, N_j)}$ by removing the bottom $k - M_j$ particles and adding infinitely many particles to the top. By comparison techniques, see \cite[Corollary 3.9, Remark 8, Remark 9]{MyOwn2}, if $j \ge j_k$, we can couple $X^{(M_j, N_j)}$ and $\ol{X}$ so that 
\begin{equation}
\label{eq:comparison-1}
Y^{(M_j, N_j)}_k(t) \ge \ol{Y}_k(t),\ t \in [0, T].
\end{equation}
Since $\ol{Y}_k$ is continuous on $[0, T]$, we can find a $u_- \in \BR$ small enough so that
\begin{equation}
\label{eq:406}
\MP\left(\min\limits_{0 \le t \le T}\ol{Y}_k(t) \ge u_-\right) \ge 1 - \frac{\eta}2.
\end{equation}
Comparing~\eqref{eq:comparison-1} and~\eqref{eq:406}, we get that for all $j \ge j_k$, 
\begin{equation}
\label{eq:comp-1}
\MP\left(\min\limits_{0 \le t \le T}Y^{(M_j, N_j)}_k(t) \ge u_-\right) \ge 1 - \frac{\eta}2.
\end{equation}
Similarly to~\eqref{eq:comp-1}, we can find a $u_+$ large enough so that for all $j \ge j_k$, we have:
\begin{equation}
\label{eq:comp-2}
\MP\left(\max\limits_{0 \le t \le T}Y^{(M_j, N_j)}_k(t) \le u_+\right) \ge 1 - \frac{\eta}2.
\end{equation}
Combining~\eqref{eq:comp-1} and~\eqref{eq:comp-2}, we get~\eqref{eq:comp-0}. 
\end{proof}

\begin{lemma}  For $j \ge j_k$, define the set of names:
\label{lemma:set-of-names}
$$
\mathcal I^{(j)}_k := \left\{i \in \MZ\mid \exists\, t \in [0, T]:\, X^{(M_j, N_j)}_i(t) = Y^{(M_j, N_j)}_k(t)\right\}.
$$
For every $\eta > 0$, there exist $I_-, I_+ \in \MZ$ and $J_k \ge 0$ such that for all $j \ge J_k$, we get:
$$
\MP\left(\mathcal I^{(j)}_k \subseteq [I_-, I_+]\right) \ge 1 - \eta.
$$
\end{lemma}

\begin{proof} Because $x \in \mathcal W$, we have: 
$$
x_i \to \infty\ \mbox{as}\ i \to \infty;\ x_i \to -\infty\ \mbox{as}\ i \to -\infty.
$$
Therefore, there exist $i_{\pm} \in \MZ$ such that for every $i \in \MZ$, 
$$
i \ge i_+\ \Ra\ x_i > u_+ + \ol{g}T;\ \mbox{and}\ i \le i_-\ \Ra\ x_i < u_- - \ol{g}T.
$$
For all $i \in \MZ$ and $j \ge j_i$, let 
$$
A_i^{(j)} := \left\{\exists\, t \in [0, T]: X^{(M_j, N_j)}_i(t) \in [u_-, u_+]\right\}.
$$
Applying \cite[Lemma 7.1]{MyOwn6} and using~\eqref{eq:drift-j}, ~\eqref{eq:diffusion-j},~\eqref{eq:X-repr},~\eqref{eq:bounded-coefficients},~\eqref{eq:series-122}, we get: for $i \ge i_+,\, j \ge j_i$, 
\begin{equation}
\label{eq:estimate-for-probab-1}
\MP\left(A_i^{(j)}\right) \le \MP\left(\min\limits_{t \in [0, T]}X_i^{(M_j, N_j)}(t) \le u_+\right) \le 2\Psi\left(\frac{x_i - u_+ - \ol{g}T}{\ol{\si}\sqrt{T}}\right).
\end{equation}
Similarly, for $i \le i_-$ and $j \ge j_i$, we have:
\begin{equation}
\label{eq:estimate-for-probab-2}
\MP\left(A_i^{(j)}\right) \le \MP\left(\max\limits_{t \in [0, T]}X_i^{(M_j, N_j)}(t) \ge u_-\right) \le 2\Psi\left(\frac{-x_i + u_- + \ol{g}T}{\ol{\si}\sqrt{T}}\right).
\end{equation}
From $x \in \mathcal W$, we have:
\begin{equation}
\label{eq:two-series}
\SL_{i \ge i_+}e^{-\al x_i^2} < \infty,\ \mbox{and}\ \SL_{i \le i_-}e^{-\al x_i^2} < \infty\ \mbox{for all}\ \al > 0.
\end{equation}
Applying \cite[Lemma 7.2]{MyOwn6} and using~\eqref{eq:two-series}, we obtain:
$$
\SL_{i \ge i_+}\Psi\left(\frac{x_i - u_+ - \ol{g}T}{\ol{\si}\sqrt{T}}\right) < \infty,\ \ \mbox{and}\ \ \SL_{i \le i_-}\Psi\left(\frac{-x_i + u_- + \ol{g}T}{\ol{\si}\sqrt{T}}\right) < \infty.
$$
Find $i'_+ > i_+$ large enough and $i'_- < i_-$ small enough so that 
\begin{equation}
\label{eq:estimate-for-series}
\SL_{i \ge i'_+}\Psi\left(\frac{x_i - u_+ - \ol{g}T}{\ol{\si}\sqrt{T}}\right) < \frac{\eta}6,
\ \ \mbox{and}\ \ \SL_{i \le i'_-}\Psi\left(\frac{-x_i + u_- + \ol{g}T}{\ol{\si}\sqrt{T}}\right) < \frac{\eta}6.
\end{equation}
Comparing~\eqref{eq:estimate-for-probab-1},~\eqref{eq:estimate-for-probab-2},~\eqref{eq:estimate-for-series}, we get: for $j \ge J_k := j_{i'_+}\vee j_{i'_-}$, 
\begin{equation}
\label{eq:probability}
\MP\Bigl(\bigcup\limits_{i = i'_+}^{N_j}A_i^{(j)}\Bigr) \le \frac{\eta}3,\ \ \mbox{and}\ \ \MP\Bigl(\bigcup\limits_{i = M_j}^{i'_-}A_i^{(j)}\Bigr) \le \frac{\eta}3.
\end{equation}
Let $I_- := i'_-+1$ and $I_+ := i'_+-1$. It follows from~\eqref{eq:probability} that for all $j \ge J_k$, with probability greater than or equal to $1 - 2\eta/3$, only the particles $X^{(M_j, N_j)}_{I_-}, \ldots, X^{(M_j, N_j)}_{I_+}$,  among the particles $X^{(M_j, N_j)}_i$, $M_j \le i \le N_j$, can ever visit the interval $[u_-, u_+]$ during time interval $[0, T]$. 
Using Lemma~\ref{lemma:comp-0}, choose $u_+$ and $u_-$ so that with probability greater than or equal to $1 - \eta/3$, the particle $Y^{(M_j, N_j)}_k$ stays within $[u_-, u_+]$ during $[0, T]$. Then with probability greater than or equal to $1 - \eta$, the ranked particle $Y^{(M_j, N_j)}_k$ can assume only the following names: $I_-, I_-+1, \ldots, I_+$. 
\end{proof}

\begin{lemma} Take the integers $I_{\pm}$ from Lemma~\ref{lemma:set-of-names}. If the following event happens: 
$$
\left\{\mathcal I^{(j)}_k \subseteq [I_-, I_+]\right\},
$$
then a.s. for every $t \in [0, T]$, we have: $Y^{(M_j, N_j)}_k(t)$ is the $(k - I_- + 1)$st bottom-ranked number among 
$$
X_{I_-}^{(M_j, N_j)}(t), \ldots, X_{I_+}^{(M_j, N_j)}(t).
$$
\label{lemma:suitable-rank}
\end{lemma}

\begin{proof} Fix a $t \in [0, T]$ such that there is no tie at time  $t$ in the system $X^{(M_j, N_j)}$. The set $\Tau$ of these $t$ has full Lebesgue measure $\mes(\cdot)$; that is, $\mes([0, T]\setminus\Tau) = 0$. Let us show that  
\begin{equation}
\label{eq:min}
\mbox{for}\ \ i = M_j, \ldots, I_- - 1,\ \ \mbox{we have:}\ \ X^{(M_j, N_j)}_i(t) < Y_k^{(M_j, N_j)}(t).
\end{equation}
Assume the converse. Recall that $X^{(M_j, N_j)}_i(0) = x_i \le Y_k^{(M_j, N_j)}(0) = x_k$. By continuity, there exists an $s \in [0, t]$ such that $X^{(M_j, N_j)}_i(s) = Y_k^{(M_j, N_j)}(s)$. This means that $i \in \mathcal I^{(j)}_k$. But $i < I_-$, and this contradicts the assumption that the event $\left\{\mathcal I^{(j)}_k \subseteq [I_-, I_+]\right\}$ has happened. This proves~\eqref{eq:min}. Similarly, we can show that 
\begin{equation}
\label{eq:max}
\mbox{for}\ \ i = I_+ + 1, \ldots, N_j,\ \ \mbox{we have:}\ \ X^{(M_j, N_j)}_i(t) > Y_k^{(M_j, N_j)}(t).
\end{equation}
We proved~\eqref{eq:min} and~\eqref{eq:max} for $t \in \Tau$; if~\eqref{eq:min} and~\eqref{eq:max} are true, then the statement Lemma~\ref{lemma:suitable-rank} holds for this $t$. But since $\mes([0, T]\setminus\Tau) = 0$, the set $\Tau$ is dense in $[0, T]$. 
Apply continuity to prove~\eqref{eq:min} and~\eqref{eq:max} for all $t \in [0, T]$ (with non-strict inequalities instead of strict ones). Because ties are resolved in lexicographic order, this completes the proof. 
\end{proof}

\begin{lemma} For every $\eps, \eta > 0$, there exists a $\de > 0$ such that for all $j \ge 1$, we have: 
\label{lemma:final-tight}
\begin{equation}
\label{eq:final-tight}
\varlimsup\limits_{j \to \infty}\MP\left(\oa\Bigl(Y^{(M_j, N_j)}_k, [0, T], \de\Bigr) \ge \eps,\ \ \mathcal I^{(j)}_k \subseteq [I_-, I_+]\right) \le \eta,
\end{equation}
\end{lemma}

\begin{proof} The sequence $\bigl(\bigl(X^{(M_j, N_j)}_{I_-}, \ldots, X^{(M_j, N_j)}_{I_+}\bigr)\bigr)_{j \ge J_k}$ is tight in $C\left([0, T], \BR^{I_+ - I_- + 1}\right)$. The mapping $C([0, T], \BR^{I_+-I_-+1}) \to C[0, T]$, which maps $(f_1,\ldots, f_{I_+-I_-+1})$ to the $K$th ranked among $f_1(t), \ldots, f_{I_+-I_-+1}(t)$, for every $t \in [0, T]$, is Lipschitz continuous. For every $j \ge J_k$ and $t \ge 0$, define $\tilde{Y}^{(j)}(t)$ to be the $(k - I_- + 1)$th bottom-ranked real number among 
$$
X^{(M_j, N_j)}_{I_-}(t), \ldots, X^{(M_j, N_j)}_{I_+}(t).
$$
Then the sequence  of stochastic processes
$$
\tilde{Y}^{(j)} = (\tilde{Y}^{(j)}(t),\, 0 \le t \le T),\ j \ge J_k,
$$
is tight in $C[0, T]$. Applying the Arzela-Ascoli criterion, we get: there exists $\de > 0$ such that 
$$
\MP\left(\oa\Bigl(\tilde{Y}^{(j)}, [0, T], \de\bigr) > \eps\right) \le \eta.
$$
Together with Lemma~\ref{lemma:suitable-rank}, this proves~\eqref{eq:final-tight}.
\end{proof}

Let us finish the proof of Lemma~\ref{lemma:tightness-of-Y}. To show tightness of $(Y^{(M_j, N_j)}_k)_{j \ge 1}$, we use the Arzela-Ascoli criterion. Fix an $\eps > 0$. We shall prove that  
\begin{equation}
\label{eq:tightness-AA}
\lim\limits_{\de \to 0}\varlimsup\limits_{j \to \infty}\MP\left(\oa\Bigl(Y^{(M_j, N_j)}_k, [0, T], \de\Bigr) \ge \eps\right) = 0.
\end{equation}
To this end, fix an $\eta > 0$ and let us show that there exists a $\de > 0$ such that 
\begin{equation}
\label{eq:tightness-refined}
\varlimsup\limits_{j \to \infty}\MP\left(\oa\Bigl(Y^{(M_j, N_j)}_k, [0, T], \de\Bigr) \ge \eps\right) \le 2\eta.
\end{equation}
But~\eqref{eq:tightness-refined} follows from Lemmata~\ref{lemma:set-of-names} and~\ref{lemma:final-tight}. This completes the proof of Lemma~\ref{lemma:tightness-of-Y}. 

\subsubsection{Proof of Lemma~\ref{lemma:no-ties}} For simplicity of notation, assume $(M'_j, N'_j) = (M_j, N_j)$. Define the event that there is a tie of finitely many particles at time $t$:
$$
E_1 = \{\exists\, k, l \in \MZ,\, k < l\ \ \mbox{such that}\ \ Y_{k-1}(t) < Y_k(t) = Y_{k+1}(t) = \ldots = Y_l(t) < Y_{l+1}(t)\}.
$$ 
Define the event that there is a tie of infinitely many particles at time $t$:
$$
E_2 = \{\exists\, w \in \BR:\ \mbox{for infinitely many}\ i \in \MZ,\ X_i(t) = w\}.
$$
Then we have:
\begin{equation}
\label{eq:big-union}
\{Y\ \mbox{has a tie at time}\ t\} = E_1\cup E_2.
\end{equation}

\smallskip

{\it Step 1.} Let us show that $\MP(E_1) = 0$. For $k, l \in \MZ$ such that $k < l$, and for $q_-, q_+ \in \MQ$, $m = 1, 2, \ldots$ define the following event:
\begin{align*}
D(k, l, q_-, q_+, m) :=& \Bigl\{Y_{k-1}(s) < q_- < Y_k(s) = Y_{k+1}(s) = \ldots = Y_l(s) < q_+ < Y_{l+1}(s)\\ & \ \ \mbox{for all}\ \ s \in [t - m^{-1}, t + m^{-1}]\Bigr\}.
\end{align*}
By continuity of trajectories of $Y_{k-1}, Y_k, \ldots, Y_{l+1}$, we can represent
\begin{equation}
\label{eq:union-1}
E_1 = \bigcup D(k, l,q_-, q_+, m),
\end{equation}
where the union in the right-hand side of~\eqref{eq:union-1} is taken over all 
\begin{equation}
\label{eq:all-parameters}
k, l \in \MZ;\ \ q_-, q_+ \in \MQ,\, q_- < q_+;\ m = 1, 2, \ldots
\end{equation}
Therefore, it suffices to show that 
\begin{equation}
\label{eq:D-event}
\MP\left(D(k, l, q_-, q_+, m)\right) = 0\ \ \mbox{for all}\ \ k, l, q_-, q_+, m\ \ \mbox{from~\eqref{eq:all-parameters}.}
\end{equation}
Assume the converse: that the probability in~\eqref{eq:D-event} is positive. If $D(k, l, q_-, q_+, m)$ happened, then for large enough $j$ we have:
$$
Y_{k-1}^{(M_j, N_j)}(s) < q_- < Y_k^{(M_j, N_j)}(s) \le Y_l^{(M_j, N_j)}(s) < q_+ < Y_{l+1}^{(M_j, N_j)}(s),\ \ s \in \left[t - m^{-1}, t + m^{-1}\right].
$$
By Lemma~\ref{lemma:cut} from Appendix, on the time interval $[t - m^{-1}, t + m^{-1}]$, the collection of random processes
$$
\left(Y_k^{(M_j, N_j)}(\cdot), \ldots, Y_l^{(M_j, N_j)}(\cdot)\right)
$$
behaves as a ranked system of $l - k + 1$ competing Brownian particles with drift coefficients $g_k, \ldots, g_l$, and diffusion coefficients $\si_k^2, \ldots, \si_l^2$, starting from the initial conditions
$$
y^{(j)} := \left(Y_k^{(M_j, N_j)}(t - m^{-1}), \ldots, Y_l^{(M_j, N_j)}(t - m^{-1})\right).
$$
We have the following convergence:
$$
\lim\limits_{j \to \infty}y^{(j)} = y^{(\infty)} := \left(Y_k(t - m^{-1}), \ldots, Y_l(t - m^{-1})\right).
$$
By Feller property given in Lemma~\ref{lemma:Feller} in Appendix, we have: On the time interval $[ t - m^{-1}, t + m^{-1}]$, the system $(Y_k, \ldots, Y_l)$ also behaves as a ranked system of $l - k + 1$ competing Brownian particles with drift coefficients $g_k, \ldots, g_l$ and diffusion coefficients $\si_k^2, \ldots, \si_l^2$, starting from $y^{(\infty)}$. But the probability that such system has a tie at any fixed time is zero, see \cite[Lemma 2.3]{MyOwn6}. 
This completes the proof of~\eqref{eq:D-event}. Combining~\eqref{eq:union-1},~\eqref{eq:D-event}, we get $\MP(E_1) = 0$. 

\smallskip

{\it Step 2.} Now, let us show that $\MP(E_2) = 0$. For $u_-, u_+ \in \BR$, introduce the event $E(u_-, u_+, k)$, which is that infinitely many particles $X_i$ visited $[u_-, u_+]$ and collided with $Y_k$ during the time interval $[0, T]$. Then we have the following representation
\begin{equation}
\label{eq:union-2}
E \subseteq \bigcup E(u_-, u_+, k),
\end{equation}
where the union is taken over all $u_-, u_+ \in \MQ$ such that $u_- < u_+$ and over all $k \in \MZ$. Let us show that
\begin{equation}
\label{eq:E-event}
\MP(E(u_-, u_+, k)) = 0\ \mbox{for all}\ u_-, u_+, k\ \ \mbox{with}\ \ u_- < u_+,\ k \in \MZ.
\end{equation}
It is straightforward to check that  
\begin{equation}
\label{eq:empty}
E(u_-, u_+, k)\cap\{\mathcal I^{(j)}_k \subseteq [I_-, I_+]\} = \varnothing.
\end{equation}
Assume $\MP(E(u_-, u_+, k)) = \zeta > 0$. Apply Lemma~\ref{lemma:set-of-names} to $\zeta$ instead of $\eta$, and arrive at a contradiction with~\eqref{eq:empty}. 
This contradiction proves~\eqref{eq:E-event}. Combining~\eqref{eq:union-2} and~\eqref{eq:E-event}, we get $\MP(E_2) = 0$. 

\subsection{Proof of Lemma~\ref{lemma:comp}} Let us show (a); (b) is similar. Take an approximative sequence $(M_j, N_j)$. Define $\ol{X}^{(M_j, N_j)}$ and $\ol{Y}^{(M_j, N_j)}$ as in Lemma~\ref{lemma:approx}, but for the system $\ol{X}$ instead of $X$. Take approximating sequences of finite systems of competing Brownian particles for each of these two-sided infinite systems. In the notation of Lemma~\ref{lemma:approx}, for every finite subset $I \subseteq \MZ$ and every $t > 0$, we have the following weak convergence:
\begin{equation}
\label{eq:conv-of-X-Y}
\bigl[Y^{(M_j, N_j)}(t)\bigr]_I  \Ra [Y(t)]_I,\ \ \bigl[\ol{Y}^{(M_j, N_j)}(t)\bigr]_I \Ra [\ol{Y}(t)]_I,\ \ j \to \infty.
\end{equation}
By comparison techniques from \cite[Corollary 3.11]{MyOwn2}, we get: 
\begin{equation}
\label{eq:comp-of-Y-333}
\bigl[Y^{(M_j, N_j)}(t)\bigr]_I \preceq \bigl[\ol{Y}^{(M_j, N_j)}(t)\bigr]_I.
\end{equation}
Combining~\eqref{eq:conv-of-X-Y} and~\eqref{eq:comp-of-Y-333} and noting that stochastic comparison is preserved under weak limits, we prove that $[Y(t)]_I \preceq [\ol{Y}(t)]_I$ for every finite subset $I \subseteq \MZ$. Therefore, $Y(t) \preceq \ol{Y}(t)$.

\subsection{Proof of Lemma~\ref{lemma:ranked}} (a) It suffices to show the following two statements:

\smallskip

(a) a.s. there exists only finitely many $n \ge 1$ such that $\min_{t \in [0, T]}X_n(t) \le u_+$;

\smallskip

(b) a.s. there exists only finitely many $n \le -1$ such that $\max_{t \in [0, T]}X_n(t) \ge u_-$.

\smallskip

Let us show (a); the proof of (b) is similar. By the Borel-Cantelli lemma, it suffices to show
\begin{equation}
\label{eq:BC-ranked}
\SL_{n=1}^{\infty}\MP\left(\min\limits_{t \in [0, T]}X_n(t) \le u_+\right) < \infty.
\end{equation}
As in the proof of Lemma~\ref{lemma:tightness-of-X}, we have: for $n \in \MZ$,
\begin{equation}
\label{eq:X-n-expr}
X_n(t) = x_n + \int_0^t\be_n(s)\,\md s + \int_0^t\rho_n(s)\,\md W_n(s),\ t \ge 0,
\end{equation}
where for all $s \ge 0$, $n \in \MZ$, 
\begin{equation}
\label{eq:bdd-coefficients}
\bigl|\be_n(s)\bigr| \le \ol{g},\ \ \bigl|\rho_n(s)\bigr| \le \ol{\si}.
\end{equation}
By \cite[Lemma 7.1]{MyOwn6}, if $n$ is such that $x_n > \ol{g}T + u_+$, then 
\begin{equation}
\label{eq:estimation-of-min}
\MP\left(\min\limits_{t \in [0, T]}X_n(t) \le u_+\right) \le 2\Psi\left(\frac{x_n - \ol{g}T - u_+}{\ol{\si}\sqrt{T}}\right).
\end{equation}
But $x \in \mathcal W$, and therefore 
$$
\SL_{n=1}^{\infty}e^{-\al x_n^2} < \infty\ \mbox{for all}\ \al > 0.
$$
Moreover, $x_n \to \infty$ as $n \to \infty$, hence there exists an $n_0$ such that $x_n > \ol{g}T + u_+$ for $n \ge n_0$. Applying \cite[Lemma 7.2]{MyOwn6}, we have:
\begin{equation}
\label{eq:series-finite}
\SL_{n=n_0}^{\infty}\Psi\left(\frac{x_n - \ol{g}T - u_+}{\ol{\si}\sqrt{T}}\right) < \infty.
\end{equation}
Combining~\eqref{eq:estimation-of-min} and~\eqref{eq:series-finite}, we get~\eqref{eq:BC-ranked}, which completes the proof of Lemma~\ref{lemma:ranked} (a).

\smallskip

(b) Similar to the proof of \cite[Lemma 3.5]{MyOwn6}; follows from Lemma~\ref{lemma:ranked} (a) and similar properties for finite systems.

\subsection{Proof of Theorem~\ref{thm:stationary}} \subsubsection{Overview of the proof} Similarly to the proof of the main result in \cite{MyOwn13}, we approximate this two-sided infinite system by finite systems of competing Brownian particles in stationary gap distributions, with suitably chosen uniformly bounded drifts. These stationary gap distributions have product-of-exponential form, which  match the infinite poduct-of-exponentials distribution $\pi_{a, b}$. Let us describe the desired approximating sequence of finite systems. These are systems of competing Brownian particles:
$$
X^{(j)} = \left(X^{(j)}_{M_j}, \ldots, X^{(j)}_{N_j}\right),\, j \ge 1,
$$
with $(M_j, N_j)$ an approximative sequence (chosen later) from Definition~\ref{defn:approx}, with $M_j \le -j < j < N_j$ for $j \ge 1$; drift coefficients (chosen later)
\begin{equation}
\label{eq:drift-choice}
g_{M_j}^{(j)}, \ldots, g_{N_j}^{(j)};
\end{equation}
and unit diffusion coefficients 
$$
\si_{M_j}^{(j)} = \ldots = \si_{N_j}^{(j)} = 1.
$$
We assume the initial conditions for each system $X^{(j)}$ are ranked, and $X_0^{(j)}(0) = 0$. 
Define the corresponding vector of ranked particles, and the gap process, respectively: 
$$
Y^{(j)} = \left(Y^{(j)}_{M_j}, \ldots, Y^{(j)}_{N_j}\right),\ \ 
Z^{(j)} = \left(Z^{(j)}_{M_j}, \ldots, Z^{(j)}_{N_j-1}\right).
$$

\begin{lemma} For each $j \ge 1$, we can choose an approximative sequence $(M_j, N_j)_{j \ge 1}$, and drift coefficients from~\eqref{eq:drift-choice}, so that the system $X^{(j)}$ has a stationary gap distribution 
$$
Z^{(j)}(t) \sim \bigotimes\limits_{k=M_j}^{N_j-1}\Exp\bigl(\la^{(j)}_k\bigr),\ \ t \ge 0,
$$
and the parameters $\la^{(j)}_k,\, k = M_j, \ldots, N_j-1,\, j \ge 1$, satisfy 
\begin{equation}
\label{eq:same}
g_k^{(j)} = g_k,\ \la^{(j)}_k = \la_k,\ -j \le k \le j.
\end{equation}
Moreover, there exist constants $C_0, C_1, C_2 > 0$ such that
\begin{equation}
\label{eq:g-bdd}
\bigl|g_k^{(j)}\bigr| \le C_0,\ \mbox{for all}\ j \ge 1,\ M_j \le k \le N_j,
\end{equation}
\begin{equation}
\label{eq:sublinear}
|\la_k| \le C_1|k| + C_2,\ \mbox{for all}\ k \in \MZ,
\end{equation}
\begin{equation}
\label{eq:sublinear-new}
\bigl|\la_k^{(j)}\bigr| \le C_1|k| + C_2,\ \mbox{for all}\ j \ge 1,\ M_j \le k < N_j.
\end{equation}
\label{lemma:bounds}
\end{lemma}

\begin{lemma}
The distribution $\pi_{a, b}$ is supported on $\mathcal V$.
\label{lemma:support}
\end{lemma}


\noindent Similarly to the proof of Lemma~\ref{lemma:approx}, we need to show the following statements. 

\begin{lemma}
\label{lemma:new-tight-X}
For every $n \in \MZ$ and $T > 0$, the sequence $(X^{(j)}_n)_{j \ge j_n}$ is tight in $C([0, T], \BR)$.
\end{lemma}

\begin{lemma}
\label{lemma:new-tight-Y}
For every $k \in \MZ$ and $T > 0$, the sequence $(Y^{(j)}_k)_{j \ge j_k}$ is tight in $C([0, T], \BR)$.
\end{lemma}

Assume that Lemmata~\ref{lemma:bounds},~\ref{lemma:support},~\ref{lemma:new-tight-X},~\ref{lemma:new-tight-Y}, are proved. Let us complete the proof of Theorem~\ref{thm:stationary}. As in the proof of Lemma~\ref{lemma:approx}, there exists an approximative subsequence $(M_{l_s}, N_{l_s})$ of $(M_j, N_j)$ such that for every finite subset $I \subseteq \MZ$ and every $T > 0$, we have:
\begin{equation}
\label{eq:Conv}
\left([X^{(l_s)}]_I, [Y^{(l_s)}]_I\right) \Ra ([X]_I, [Y]_I),\ \mbox{in}\ C([0, T], \BR^{2|I|}),\ \ s \to \infty.
\end{equation}
Here, $X = (X_i)_{i \in \MZ}$ is a two-sided infinite system of competing Brownian particles with drift coefficients $g_n,\ n \in \MZ$ (we have these drift coefficient because of~\eqref{eq:same}), and unit diffusion coefficients, and $Y = (Y_k)_{k \in \MZ}$ is its corresponding system of ranked particles. From~\eqref{eq:Conv}, for every $k \ge 1$,
\begin{equation}
\label{eq:Conv-Gap}
\left(Z^{(l_s)}_{-k}, \ldots, Z^{(l_s)}_k\right) \Ra \left(Z_{-k}, \ldots, Z_k\right)\ \mbox{in}\ C([0, T], \BR^{2k+1})\ \mbox{as}\ s \to \infty. 
\end{equation}
For every $t \ge 0$ and $s$ large enough so that $l_s \ge k$, we have:
\begin{equation}
\label{eq:finite-approx}
\left(Z^{(l_s)}_{-k}(t),\ldots, Z^{(l_s)}_{k}(t)\right) \sim \bigotimes\limits_{m=-k}^{k}\Exp\left(\la_m\right).
\end{equation}
Combining~\eqref{eq:Conv-Gap} with~\eqref{eq:finite-approx}, we have:  for $t \ge 0$ and $k \ge 1$, 
$$
(Z_{-k}(t), \ldots, Z_{k}(t)) \sim \bigotimes\limits_{m=-k}^{k}\Exp\left(\la_m\right).
$$
Thus $Z(t) \sim \pi_{a, b}$ for all $t \ge 0$. This completes the proof of Theorem~\ref{thm:stationary}.

\subsubsection{Proof of Lemma~\ref{lemma:bounds}} By Remark~\ref{rmk:BVP} from Appendix, the sequence $(\la^{(j)}_k)_{M_j \le k < N_j}$ is a unique solution to the following  difference equation similar to~\eqref{eq:difference}, 
\begin{equation}
\label{eq:difference-finite}
\frac12\la_{k-1}^{(j)} - \la_k^{(j)} + \frac12\la_{k+1}^{(j)} = g_{k+1}^{(j)} - g_k^{(j)},\ \ k = M_j, \ldots, N_j - 1,
\end{equation}
together with added boundary conditions
\begin{equation}
\label{eq:BV}
\la^{(j)}_{M_j-1} = \la^{(j)}_{N_j} = 0.
\end{equation}
Assume that, for some parameters $c^{\pm}_j$ to be determined later,   
\begin{equation}
\label{eq:special-tails}
g^{(j)}_{j+1} = \ldots = g^{(j)}_{N_j} = c_j^+,\ g^{(j)}_{M_j} = \ldots = g^{(j)}_{-j-1} = c^-_j.
\end{equation}
Knowing~\eqref{eq:same}, ~\eqref{eq:BV},~\eqref{eq:difference-finite},
~\eqref{eq:special-tails}, let us solve for $\la^{(j)}_k,\, j < k < N_j$, and $c^+_j$. We have:
$$
\frac12\la_{k-1}^{(j)} - \la_k^{(j)} + \frac12\la_{k+1}^{(j)} = 0,\ \ k = j+1, \ldots, N_j - 1.
$$
Therefore, $(\la_j^{(j)}, \ldots, \la_{N_j}^{(j)})$ is a linear sequence (arithmetic progression). Together with the second equality in~\eqref{eq:BV}, this means 
\begin{equation}
\label{eq:linear-lambda}
\la_{k}^{(j)} = (N_j - k)\la_{N_j - 1}^{(j)},\ \ k = j, \ldots, N_j.
\end{equation}
In particular, letting $k = j$ in~\eqref{eq:linear-lambda}, and applying~\eqref{eq:same}, we get:
\begin{equation}
\label{eq:932}
\la_j = (N_j - j)\la_{N_j - 1}^{(j)}.
\end{equation}
Comparing $\la_{j}^{(j)}$ and $\la_{j+1}^{(j)}$ from~\eqref{eq:linear-lambda} and~\eqref{eq:932}, we get:
\begin{equation}
\label{eq:Ratio}
\la_{j+1}^{(j)} = \frac{m_j - 1}{m_j}\la_j,\ \ m_j := N_j - j.
\end{equation}
From~\eqref{eq:same}, we get: $\la_{j-1}^{(j)} = \la_{j-1}$. Plug $k = j$ into~\eqref{eq:difference-finite} and get:
\begin{equation}
\label{eq:1133}
\frac12\la_{j-1} - \la_j + \frac{m_j-1}{2m_j}\la_j = c_j^+ - g_j.
\end{equation}
Solve~\eqref{eq:1133} for $c_j^+$: 
\begin{equation}
\label{eq:1139}
c_j^+ = -\frac12\left(\la_j - \la_{j-1}\right) - \frac1{2m_j}\la_j + g_j.
\end{equation}
From~\eqref{eq:rates} and~\eqref{eq:bdd-seq}, it is easy to see that 
$$
\sup\limits_{j \in \MZ}\left|\la_j - \la_{j-1}\right| < \infty.
$$
It suffices to take $m_j$ large enough, say $m_j \ge \la_j$ (or, equivalently, $N_j \ge j + \la_j$), to make the right-hand side of~\eqref{eq:1139} bounded. Thus, we can ensure that 
\begin{equation}
\label{eq:bounded-c-plus}
\sup\limits_{j \ge 1}|c^+_j| < \infty.
\end{equation}
Similarly, by a suitable choice of $c^-_j$ we can ensure that 
\begin{equation}
\label{eq:bounded-c-minus}
\sup\limits_{j \ge 1}|c^-_j| < \infty.
\end{equation}
Using~\eqref{eq:bounded-c-plus},~\eqref{eq:bounded-c-minus}, and $\sup_{n \in \MZ}|g_n| < \infty$, it is easy to check that~\eqref{eq:g-bdd} holds:
$$
\sup\limits_{j, k}\bigl|g^{(j)}_k\bigr| \le \max\bigl(\sup\limits_{j \ge 1}\bigl|c_j^+\bigr|,\ \sup\limits_{j \ge 1}\bigl|c_j^-\bigr|,\ \sup\limits_{k \in \MZ}|g_k|\bigr) =: C_0 < \infty.
$$
Thus we constructed a required sequence of finite systems of competing Brownian particles which satisfies~\eqref{eq:same} and~\eqref{eq:g-bdd}. The estimate~\eqref{eq:sublinear} follows immediately from~\eqref{eq:rates}, combined with~\eqref{eq:bdd-seq}. Next, apply~\eqref{eq:difference-mu} from Appendix to our system: For $k \ge 0$, we get:
\begin{equation}
\label{eq:representation-of-lambda}
\la^{(j)}_k = \la_0^{(j)} - 2k\ol{g}^{(j)} + 2\bigl(g^{(j)}_{1} + \ldots + g^{(j)}_{k}\bigr),\ \ 
\end{equation}
\begin{equation}
\label{eq:average}
\mbox{where}\ \  \ol{g}^{(j)} := \frac1{N_j - M_j + 1}\bigl(g^{(j)}_{M_j} + \ldots + g^{(j)}_{N_j}\bigr).
\end{equation}
It follows from~\eqref{eq:g-bdd} and~\eqref{eq:average} that
\begin{equation}
\label{eq:average-bdd}
\sup\limits_{j \ge 1}\left|\ol{g}^{(j)}\right| \le C_0 < \infty.
\end{equation}
Note that $\la_0^{(j)} = \la_0$ for all $j \ge 1$. Combining~\eqref{eq:representation-of-lambda} with~\eqref{eq:g-bdd} and~\eqref{eq:average-bdd}, we get:
$$
\bigl|\la^{(j)}_k\bigr| \le |\la_0| + 4|k|C_0.
$$
This proves~\eqref{eq:sublinear-new}. The case $k \le 0$ is treated similarly. 

\subsubsection{Proof of Lemma~\ref{lemma:support}} Let $z \sim \pi_{a, b}$, and let $x := \Phi(z)$. From~\eqref{eq:V-W}, we have: $z \in \mathcal V$ if and only if $x \in \mathcal W$. To show $x \in \mathcal W$ a.s., we need to prove the two following statements:
\begin{equation}
\label{eq:series-plus}
\SL_{n \ge 1}e^{-\al x_n^2} < \infty\ \mbox{a.s. for all}\ \al > 0,
\end{equation}
\begin{equation}
\label{eq:series-minus}
\SL_{n \le -1}e^{-\al x_n^2} < \infty\ \mbox{a.s. for all}\ \al > 0.
\end{equation}
Let us show~\eqref{eq:series-plus}; ~\eqref{eq:series-minus} is similar. Use that 
\begin{equation}
\label{eq:sum-of-z}
x_n = z_0 + \ldots + z_{n-1},\, n \ge 1.
\end{equation}
From the estimate~\eqref{eq:sublinear}, we have:
$$
\bigotimes\limits_{n=0}^{\infty}\Exp\left(C_1 + C_2n\right) \preceq \bigotimes\limits_{n=0}^{\infty}\Exp\left(\la_n\right) \sim z := (z_n)_{n \ge 1}.
$$
Therefore, we can find independent $\tilde{z}_n \sim \Exp(C_1 + C_2n),\, n \ge 0$, such that 
\begin{equation}
\label{eq:coupling}
z_n \ge \tilde{z}_n\ \mbox{for all}\ n \ge 0.
\end{equation}
Comparing~\eqref{eq:sum-of-z} and~\eqref{eq:coupling}, we get:
\begin{equation}
\label{eq:comparing-sums}
x_n = z_0 + \ldots + z_{n-1} \ge \tilde{x}_n := \tilde{z}_0 + \ldots + \tilde{z}_{n-1},\, n \ge 1.
\end{equation}
Take an $\al > 0$ and apply~\eqref{eq:comparing-sums} to the sum in~\eqref{eq:series-plus}:
\begin{equation}
\label{eq:comparing-series}
\SL_{n = 1}^{\infty}e^{-\al x_n^2} \le \SL_{n = 1}^{\infty}e^{-\al\tilde{x}_n^2}.
\end{equation}
Apply Lemma~\ref{lemma:nice} to $(\tilde{x}_n)_{n \ge 1}$. Together with~\eqref{eq:comparing-series}, this completes the proof of Lemma~\ref{lemma:support}. 

\subsubsection{Proof of Lemma~\ref{lemma:new-tight-X}} Similar to Lemma~\ref{lemma:tightness-of-X}, except the following observation: Initial conditions $X^{(j)}_k(0),\, k \in \MZ$, are in general dependent on $j$. Recall that initial conditions of each system $X^{(j)}$ are ranked. That is, $X^{(j)}_k(0) = Y^{(j)}_k(0)$ for all $k \in \MZ$ and $j \ge 1$. To adjust the proof of Lemma~\ref{lemma:tightness-of-X}, we need only to show the following statement.  

\begin{lemma}
Fix a $k \in \MZ$ and take a $j \ge |k|$. Then the distribution of $X^{(j)}_k(0) = Y^{(j)}_k(0)$ is independent of $j$. 
\end{lemma}

\begin{proof} Fix a $j \ge 1$. Assume without loss of generality that $k > 0$. Since $Y^{(j)}_0(0) = 0$, we have:
\begin{equation}
\label{eq:sum-z-initial}
Y^{(j)}_k(0) = z^{(j)}_0 + \ldots + z^{(j)}_{k-1},\ \ n \ge 0.
\end{equation} 
Here, we consider the following independent random variables:
\begin{equation}
\label{eq:initial-gaps}
z^{(j)}_i \sim \Exp\bigl(\la_i^{(j)}\bigr),\, M_j \le i < N_j.
\end{equation}
But $\la_i^{(j)} = \la_i$ for $i = 0, \ldots, k-1$, if $j \ge k$. Therefore, the distribution of $z^{(j)}_0 + \ldots + z^{(j)}_{k-1}$ is independent of $j \ge k$, which together with~\eqref{eq:sum-z-initial} for $n := k$ proves independence of the distribution of $Y^{(j)}_k(0)$ of $j \ge |k|$. For each $j \ge 1$, the initial conditions of the system $X^{(j)}$ are ranked, that is, $X^{(j)}_n(0) = Y^{(j)}_n(0)$ for all $n \in \MZ$. In addition,  $X^{(j)}_0(0) = 0$. This completes the proof.
\end{proof}

\subsubsection{Proof of Lemma~\ref{lemma:new-tight-Y}} This is similar to the proof of Lemma~\ref{lemma:tightness-of-Y}. However, the systems $X^{(j)}$ do not start from the same initial conditions; this is their main difference from the systems $X^{(M_j, N_j)}$ from Lemma~\ref{lemma:tightness-of-Y}.  Therefore, we need to modify Lemmata~\ref{lemma:comp-0} and~\ref{lemma:set-of-names}. Fix a $k \in \MZ$.

\begin{lemma}
\label{lemma:modified-comp-0}
For every $\eta > 0$, there exist $u_{\pm} \in \BR$ such that for every $j \ge |k|$, we have:
\begin{equation}
\label{eq:new-comp-0}
\MP\left(\forall \ t \in [0, T],\ \ u_- \le Y^{(j)}_k(t) \le u_+\right) \ge 1 - \eta.
\end{equation}
\end{lemma}

\begin{proof} From~\eqref{eq:sum-z-initial}, we get:
$$
Y^{(j)}_n(0) = Y^{(j)}_k(0) + z^{(j)}_k + \ldots + z^{(j)}_{n-1},\ \ n \ge k.
$$
It follows from~\eqref{eq:initial-gaps} and the estimate~\eqref{eq:sublinear-new} that we can generate independent random variables 
\begin{equation}
\label{eq:new-initial-gaps}
\tilde{z}_n \sim \Exp(C_1 + C_2|n|),\ \mbox{such that a.s.}\ \tilde{z}_n \le z_n,\, n \ge k.
\end{equation}
Define for $j \ge |k|$ and $n \ge k$:
\begin{equation}
\label{eq:adjusted-initial}
\tilde{x}_n := X^{(j)}_k(0) + \tilde{z}_k + \ldots + \tilde{z}_{n-1}.
\end{equation}
Consider a one-sided infinite system $\tilde{X} = (\tilde{X}_n)_{n \ge k}$ of competing Brownian particles with drift coefficients $\tilde{g}_n := -C_0,\, n \ge k$, where $C_0$ is taken from~\eqref{eq:g-bdd}; unit diffusion coefficients $\tilde{\si}_n = 1,\, n \ge k$; starting from $\tilde{X}_n(0) = \tilde{x}_n,\, n \ge k$. By Lemma~\ref{lemma:nice}, $(\tilde{x}_n)_{n \ge k}$ satisfies
\begin{equation}
\label{eq:946}
\SL_{n=k}^{\infty}e^{-\al\tilde{x}_n^2} < \infty\ \mbox{a.s. for all}\ \al > 0.
\end{equation}
Therefore, by \cite[Theorem 2.1]{MyOwn6} there exists in the weak sense a unique in law version of this one-sided infinite system $\tilde{X}$. Denote by $\tilde{Y} = (\tilde{Y}_n)_{n \ge k}$ the corresponding system of ranked particles, and assume it has ranked initial conditions. From~\eqref{eq:new-initial-gaps} and~\eqref{eq:adjusted-initial}, we have: 
\begin{equation}
\label{eq:shift-down-initial}
\tilde{Y}_n(0) \le Y^{(j)}_n(0),\ \ j \ge |k|,\ \ k \le n \le N_j.
\end{equation}
By comparison techniques, \cite{MyOwn2, MyOwn6}, we obtain:
$$
\tilde{Y}_n(t) \le Y^{(j)}_n(t),\ t \ge 0,\ j \ge j_k,\ k \le n \le N_j.
$$
Indeed, the system $\tilde{X}$ is obtained from $X^{(j)}$ via: (a) removing particles with ranks less than $k$ from the bottom; (b) adding (infinitely many) particles with ranks greater than $N_j$ to the top; (c) shifting down ranked initial conditions, as in~\eqref{eq:shift-down-initial}; (d) taking smaller values $\tilde{g}_n$ of drift coefficients, by~\eqref{eq:g-bdd}. The rest of the proof of Lemma~\ref{lemma:modified-comp-0} is as in Lemma~\ref{lemma:comp-0}. 
\end{proof}

\begin{lemma}  For $j \ge |k|$, define the set of names:
\label{lemma:modified-set-of-names}
$$
\mathcal J^{(j)}_k := \left\{i \in \MZ\mid \exists\, t \in [0, T]:\ \tilde{X}^{(j)}_i(t) = \tilde{Y}^{(j)}_k(t)\right\}.
$$
For every $\eta > 0$, there exist $J_-, J_+ \in \MZ$ and $J_0 \ge 0$ such that for all $j \ge J_0$, we get:
$$
\MP\left(\mathcal J^{(j)}_k \subseteq [J_-, J_+]\right) \ge 1 - 2\eta.
$$
\end{lemma}

\begin{proof} We use the notation from the proof of Lemma~\ref{lemma:modified-comp-0}. For $j \ge j_k$ and $M_j \le n \le N_j$, let $x^{(j)}_n := X_n^{(j)}(0)$; then we can compare:
\begin{equation}
\label{eq:comp-705}
x^{(j)}_n = z^{(j)}_k + \ldots + z^{(j)}_{n-1} \ge \tilde{z}_k + \ldots + \tilde{z}_{n-1} =: \tilde{x}_n.
\end{equation}
From~\eqref{eq:946}, we have: $\tilde{x}_n \to \infty,\  n \to \infty$. Therefore, there exists an $n_0 \in \MZ$ such that for every $n \ge n_0$, we have: $\tilde{x}_n > u_+ + \ol{g}T$. From~\eqref{eq:comp-705}, we get: $x^{(j)}_n > u_+ + \ol{g}T$. In the notation of the proof of Lemma~\ref{lemma:set-of-names}, the estimate in~\eqref{eq:estimate-for-probab-1} takes the form
\begin{equation}
\label{eq:estimate-for-probab-new}
\MP\left(A_i^{(j)}\right) \le \MP\left(\min\limits_{t \in [0, T]}\tilde{X}_i^{(j)}(t) \le u_+\right) \le 2\Psi\left(\frac{\tilde{x}_i - u_+ - \ol{g}T}{\ol{\si}\sqrt{T}}\right).
\end{equation}
From~\eqref{eq:946} and \cite[Lemma 7.2]{MyOwn6}, we get that 
\begin{equation}
\label{eq:947}
\SL_{n = n_0}^{\infty}\Psi\left(\frac{\tilde{x}_n - u_+ - \ol{g}T}{\ol{\si}\sqrt{T}}\right) < \infty.
\end{equation}
Combining~\eqref{eq:947} with~\eqref{eq:estimate-for-probab-new}, we complete the proof of Lemma~\ref{lemma:modified-set-of-names} as in the proof of  Lemma~\ref{lemma:set-of-names}. 
\end{proof}

\subsection{Proof of Lemma~\ref{lemma:comparison-of-stationary-gap-distributions}}  Take versions of systems $X^{(M_j, N_j)}$ and $X^{(M_{j'}, N_{j'})}$, starting from 
$$
X^{(M_j, N_j)}_i(0) = X^{(M_{j'}, N_{j'})}_i(0) = 0\ \mbox{for all}\ i.
$$
The system $X^{(M_j, N_j)}$ is obtained from $X^{(M_{j'}, N_{j'})}$ by removing the top $N_{j'} - N_j$ particles and the bottom $M_j - M_{j'}$ particles. By comparison techniques, see \cite[Corollary 3.10]{MyOwn2}, we have: 
\begin{equation}
\label{eq:comparison-of-gap-processes}
\left[Z^{(M_{j'}, N_{j'})}(t)\right]_I \preceq \left[Z^{(M_{j}, N_{j})}(t)\right]_I.
\end{equation}
By \cite[Proposition 2.2]{MyOwn6}, we get:
\begin{equation}
\label{eq:conv-of-two-gaps}
Z^{(M_{j'}, N_{j'})}(t) \Ra \pi^{(j)},\ \ Z^{(M_{j'}, N_{j'})}(t) \Ra \pi^{(j')},\ \ t \to \infty.
\end{equation}
Combine~\eqref{eq:comparison-of-gap-processes} and~\eqref{eq:conv-of-two-gaps},  and observe that stochastic comparison is preserved under weak limits. The rest of the proof of Lemma~\ref{lemma:comparison-of-stationary-gap-distributions} is omitted. 

\subsection{Proof of Theorem~\ref{thm:no-unit-diffusions}} (a) It suffices to prove that for every $k \in \MZ$, the family $(Z_k(t),\, t \ge 0)$ is tight in $\BR$. Take a $j \ge j_k$ and a system $X^{(M_j, N_j)}$, starting from 
$$
X_n^{(M_j, N_j)}(0) = X_n(0),\ M_j \le n \le N_j.
$$
Then the corresponding gap process $Z^{(M_j, N_j)}$ corresponds to a tight family of random variables $\left(Z^{(M_j, N_j)}(t),\, t \ge 0\right)$ in $\BR^{N_j - M_j}$, by \cite[Proposition 2.2]{MyOwn6}. Therefore, the family
\begin{equation}
\label{eq:new-tight-sequence}
\bigl(Z^{(M_j, N_j)}_k(t),\, t \ge 0\bigr)
\end{equation}
is tight in $\BR$. Now, the system $X^{(M_j, N_j)}$ can be obtained from $X$ by removing top particles (with ranks greater than $N_j$) and bottom particles (with ranks less than $M_j$). Therefore, by comparison techniques from \cite[Corollary 3.10]{MyOwn2}, for every subset $I \subseteq \{M_j, \ldots, N_j - 1\}$, we get:
\begin{equation}
\label{eq:comparison-of-gaps}
0 \le \bigl[Z(t)\bigr]_I \preceq \bigl[Z^{(M_j, N_j)}(t)\bigr]_I,\ \ t \ge 0.
\end{equation}
In particular, letting $I = \{k\}$ for a $k \in \MZ$, we get from~\eqref{eq:comparison-of-gaps}:
\begin{equation}
\label{eq:comparison-of-gaps-k}
0 \le Z_k(t) \preceq Z^{(M_j, N_j)}_k(t),\ \ t \ge 0.
\end{equation}
Combining~\eqref{eq:comparison-of-gaps-k} with tightness of the family~\eqref{eq:new-tight-sequence}, we complete the proof of Theorem~\ref{thm:no-unit-diffusions} (a). 

\smallskip

(b) Take a sequence $(t_l)_{l \ge 1}$ of positive numbers such that $t_l \uparrow \infty$. Assume $Z(t_l) \Ra \nu$ for some probability measure $\mu$ on $\BR^{\MZ}_+$. Take a finite subset $I \subseteq \MZ$. It suffices to show that 
\begin{equation}
\label{eq:comparison-of-marginals}
\bigl[\nu\bigr]_I \preceq \bigl[\pi^{(\infty)}\bigr]_I.
\end{equation}
Because $(M_j, N_j)$ is an approximative sequence, we have: $M_j \to -\infty$ and $N_j \to \infty$ as $j \to \infty$. Take a $j$ large enough so that $I \subseteq \{M_j, \ldots, N_j - 1\}$, and consider a system $X^{(M_j, N_j)}$ as in the proof of (a) above. Plugging $t := t_l$ in~\eqref{eq:comparison-of-gaps}, we have: 
\begin{equation}
\label{eq:comparison-of-gap-marginals}
[Z(t_l)]_I \preceq \left[Z^{(M_j, N_j)}(t_l)\right]_I.
\end{equation}
From \cite[Proposition 2.2]{MyOwn6}, applied to marginals corresponding to the subset $I$, we have:
\begin{equation}
\label{eq:conv-629}
\left[Z^{(M_j, N_j)}(t_l)\right]_I \Ra \left[\pi^{(j)}\right]_I,\ \ l \to \infty.
\end{equation}
Since $Z(t_l) \Ra \nu$ as $l \to \infty$, we have:
\begin{equation}
\label{eq:conv-630}
[Z(t_l)]_I \Ra [\nu]_I,\ \ l \to \infty.
\end{equation}
Compare~\eqref{eq:comparison-of-gap-marginals},~\eqref{eq:conv-629},~\eqref{eq:conv-630}, and observing that stochastic comparison is preserved under weak limits, we prove~\eqref{eq:comparison-of-marginals}. This, in turn, completes the proof of Theorem~\ref{thm:no-unit-diffusions} (b).

\subsection{Proof of Theorem~\ref{thm:conv-general}} (a) Similar to the proof of Theorem~\ref{thm:stationary}: For each $j \ge 1$, we construct a finite system of competing Brownian particles 
$$
\tilde{X}^{(j)} = \left(\tilde{X}^{(j)}_{M_j}, \ldots, \tilde{X}^{(j)}_{N_j}\right)
$$
with drift and diffusion coefficients $g_n,\, \si_n^2,\, M_j \le n \le N_j$, with ranked initial conditions, and with $\tilde{X}^{(j)}_0(0) = 0$. Without loss of generality, we assume $M_1 < 0 < N_1$. For each system $\tilde{X}^{(j)}$, denote the corresponding system of ranked particles and the gap process by
$$
\tilde{Y}^{(j)} = \left(\tilde{Y}^{(j)}_{M_j}, \ldots, \tilde{Y}^{(j)}_{N_j}\right)\ \ \mbox{and}\ \ \tilde{Z}^{(j)} = \left(\tilde{Z}^{(j)}_{M_j}, \ldots, \tilde{Z}^{(j)}_{N_j - 1}\right).
$$
We assume that the gap process is in its stationary distribution:
\begin{equation}
\label{eq:stat-gap-j}
\tilde{Z}^{(j)}(t) \sim \pi^{(j)},\ \ t \ge 0.
\end{equation}

Next, we prove as in Lemma~\ref{lemma:approx} that for every finite subset $I \subseteq \MZ$ and every $T > 0$, we have the following weak convergence in $C([0, T], \BR^{2|I|})$: 
\begin{equation}
\label{eq:conv-all}
\bigl(\bigl[\tilde{X}^{(j)}\bigr]_I,\, \bigl[\tilde{Y}^{(j)}\bigr]_I\bigr) \Ra 
([X]_I, [Y]_I),\ \ j \to \infty. 
\end{equation}
As in the proof of Theorem~\ref{thm:stationary}, we combine~\eqref{eq:stat-gap-j} with~\eqref{eq:conv-all} and complete the proof. We  need only to modify Lemmata~\ref{lemma:modified-comp-0} and~\ref{lemma:modified-set-of-names}. Fix a $k \in \MZ$.

\begin{lemma}
\label{lemma:modified-comp-new}
For every $\eta > 0$, there exist $u_{\pm} \in \BR$ such that for every $j \ge |k|$, we have:
\begin{equation}
\label{eq:new-comp-0}
\MP\left(\forall \ t \in [0, T],\ \ u_- \le \tilde{Y}^{(j)}_k(t) \le u_+\right) \ge 1 - \eta.
\end{equation}
\end{lemma}

\begin{proof} Similar to that of Lemma~\ref{lemma:modified-comp-0}. We have: 
$$
\tilde{Y}^{(j)}_k(0) = \tilde{Y}^{(j)}_0(0) + z^{(j)}_0 + \ldots + z^{(j)}_{k-1} \to  z^{(\infty)}_0 + \ldots + z^{(\infty)}_{k-1},\ j \to \infty.
$$
Therefore, there exists a $y_0 \in \BR$ such that $\tilde{Y}^{(j)}_k(0) \ge y_0$ for all $j \ge |k|$. Now, 
\begin{equation}
\tilde{Y}^{(j)}_n(0) = \tilde{Y}^{(j)}_k(0) + z^{(j)}_k + \ldots + z^{(j)}_{n-1} \ge  y_0 + z^{(j)}_k + \ldots + z^{(j)}_{n-1},\ \ n \ge k.
\label{eq:initial-sums-new}
\end{equation}
Take a one-sided infinite system $\ol{X} = (\ol{X}_n)_{n \ge k}$ of competing Brownian particles with drift coefficients $(g_n)_{n \ge k}$, diffusion coefficients $(\si_n^2)_{n \ge k}$, and initial conditions 
\begin{equation}
\label{eq:1109}
\ol{X}_n(0) := y_0 + z_k^{(\infty)} + \ldots + z^{(\infty)}_{n-1},\ \ n \ge k.
\end{equation}
This system exists in the weak sense and is unique in law, because 
\begin{equation}
\label{eq:1112}
\SL_{n=k}^{\infty}e^{-\al\left[\ol{X}_n(0)\right]^2} < \infty\ \mbox{a.s. for all}\ \al > 0.
\end{equation}
Denote by $\ol{Y} = (\ol{Y}_n)_{n \ge k}$ the corresponding system of ranked particles. 
From~\eqref{eq:initial-sums-new},~\eqref{eq:1109}, we have: 
\begin{equation}
\label{eq:1110}
\ol{X}_n(0) \le \tilde{Y}^{(j)}_n(0),\ \ j \ge |k|,\ \ n \ge k.
\end{equation}
By comparison techniques, \cite{MyOwn2, MyOwn6}, we obtain:
$$
\ol{Y}_n(t) \le \tilde{Y}^{(j)}_n(t),\ t \ge 0,\ j \ge j_k,\ k \le n \le N_j.
$$
Indeed, the system $\ol{X}$ is obtained from $X^{(j)}$ via: (a) removing particles with ranks less than $k$ from the bottom; (b) adding (infinitely many) particles with ranks greater than $N_j$ to the top; (c) shifting down ranked initial conditions, as in~\eqref{eq:1110}. The rest of the proof of Lemma~\ref{lemma:modified-comp-0} is as in Lemma~\ref{lemma:comp-0}. 
\end{proof}

\begin{lemma}  For $j \ge |k|$, define the set of names:
\label{lemma:modified-set-of-names-new}
$$
\mathcal J^{(j)}_k := \left\{i \in \MZ\mid \exists\, t \in [0, T]:\ \tilde{X}^{(j)}_i(t) = \tilde{Y}^{(j)}_k(t)\right\}.
$$
For every $\eta > 0$, there exist $J_-, J_+ \in \MZ$, and $J_0 \ge 0$ such that for all $j \ge J_0$, we get:
$$
\MP\left(\mathcal J^{(j)}_k \subseteq [J_-, J_+]\right) \ge 1 - 2\eta.
$$
\end{lemma}

\begin{proof} Similar to the proof of Lemma~\ref{lemma:modified-set-of-names}, except that the role of $\ol{x} = (\ol{x}_n)_{n \ge k}$ is played by~\eqref{eq:1109}, which satisfies~\eqref{eq:1112}.   
\end{proof}

\smallskip

(b) Take another copy $\ol{X}$ of the two-sided infinite system $X$ of competing Brownian particles, with the same drift coefficients $g_n$ and diffusion coefficients $\si_n^2$, but starting from a different initial condition: 
$$
\ol{Z}(t) \sim \pi^{(\infty)},\ \ \mbox{for every}\ \ t \ge 0,
$$
where $\ol{Z}$ is the corresponding gap process. Then $\ol{Z}(0) \preceq Z(0)$. By Lemma~\ref{lemma:comp} (b), 
\begin{equation}
\label{eq:comp-644}
\ol{Z}(t) \preceq Z(t),\  \mbox{for every}\ t \ge 0.
\end{equation}
By Theorem~\ref{thm:no-unit-diffusions} (a), the family $(Z(t),\, t \ge 0)$, is tight in $\BR^{\infty}_+$. Take a weak limit point $\nu$: assume $t_l \uparrow \infty$ is a sequence of positive numbers, and $Z(t_l) \Ra \nu$. Substitute $t := t_l$ into~\eqref{eq:comp-644}, and take weak limits as $l \to \infty$. Since weak convergence preserves stochastic comparison, $\pi^{(\infty)} \preceq \nu$. On the other hand, by Theorem~\ref{thm:no-unit-diffusions} (b) $\nu \preceq \pi^{(\infty)}$. Thus, $\nu = \pi^{(\infty)}$. We proved that the family $(Z(t),\, t \ge 0)$ is tight, and any weak limit point as $t \to \infty$ is equal to $\pi^{(\infty)}$. This completes the proof of part (b). 

\subsection{Proof of Lemma~\ref{lemma:rate-limits}} First, let us show that the sequence $(\la^{(j)}_k)$ is nondecreasing. For $\si_n \equiv 1$, we can use the notation from subsection 2.3. Because 
$$
z_k^{(j+1)} \sim \Exp(\la_k^{(j+1)}) \le z_k^{(j)} \sim \Exp(\la_k^{(j)}),\ \ j \ge j_k,\ \ k \in \MZ.
$$
we get: $\la_k^{(j)} \le \la_k^{(j+1)}$. Next, from~\eqref{eq:difference-mu} applied to the current system, we get:
\begin{equation}
\label{eq:lambda-diff}
\la^{(j)}_{k+1} - \la^{(j)}_k = \ol{g}^{(j)} + g^{(j)}_k.
\end{equation}
Combining~\eqref{eq:g-bdd},~\eqref{eq:average-bdd},~\eqref{eq:lambda-diff}, we get:
$$
\sup\limits_{k, j}\bigl|\la^{(j)}_{k+1} - \la^{(j)}_k\bigr| < \infty.
$$
Therefore, as $j \to \infty$, either both limits $\la_k^{(\infty)} = \lim\la_k^{(j)}$ and $\la_{k+1}^{(\infty)} = \lim\la_{k+1}^{(j)}$ are finite, or both are infinite. This completes the proof of Lemma~\ref{lemma:rate-limits}. 

\subsection{Proof of Theorem~\ref{thm:conv}} (a) By Remark~\ref{rmk:BVP}, it suffices to show that
\begin{equation}
\label{eq:goal}
\frac12\la^{(\infty)}_{k-1} - \la^{(\infty)}_k + \frac12\la^{(\infty)}_{k+1} = g_{k+1} - g_k,\ \ k \in \MZ.
\end{equation}
Applying~\eqref{eq:difference} from the Appendix to the system $X^{(j)}$, we get:
\begin{equation}
\label{eq:have}
\frac12\la^{(j)}_{k-1} - \la^{(j)}_k + \frac12\la^{(j)}_{k+1} = g_{k+1} - g_k,\ \ M_j + 1 \le k \le N_j - 2,\ \ j \ge 1.
\end{equation}
Combining~\eqref{eq:limits},~\eqref{eq:have}, we get~\eqref{eq:goal}. Apply Theorem~\ref{thm:stationary} to finish the proof of Theorem~\ref{thm:conv} (a).

\medskip

(b, c) Immediately follow from Theorems~\ref{thm:no-unit-diffusions} and~\ref{thm:conv-general}. 

\subsection{Proof of Theorem~\ref{thm:conv-0}} We have: $\pi^{(\infty)} = \de_{\mathbf{0}}$. Every weak limit point $\nu$ of $Z(t)$ as $t \to \infty$ is stochastically dominated by $\de_{\mathbf{0}}$. Since $\nu$ is  supported on $\BR^{\infty}_+$,  it is equal to $\de_{\mathbf{0}}$. Therefore, every weak limit point $\nu$ of the family $(Z(t),\, t \ge 0)$, as $t \to \infty$, is equal to $\de_{\mathbf{0}}$. Combining this with tightness of $(Z(t),\, t \ge 0)$ in $\BR^{\MZ}_+$ from Theorem~\ref{thm:no-unit-diffusions} (a), we complete the proof.

\section{Appendix}

\begin{lemma}
\label{lemma:nice} 
Fix $c_1, c_2 > 0$, $k \in \MZ$. Consider a sequence $z := (z_n)_{n \ge k}$ of independent random variables $z_n \sim \Exp(c_1 + c_2|n|)$.  Fix an $x_k \in \BR$ and define the sequence $(x_n)_{n \ge k}$ as follows: 
$$
x_n := x_k + z_k + \ldots + z_{n-1},\ \ n \ge k.
$$
Then a.s. for every $\al > 0$ we have: 
$$
\SL_{n=k}^{\infty}e^{-\al x_n^2} < \infty.
$$ 
\end{lemma}

\begin{proof} Let $\la_n := c_1 + c_2|n|,\, n \ge k$. Then $\sum_{n \ge 1}\la_n^{-2} < \infty$, and the numbers 
$$
\La_n := \SL_{j=k}^n\la_j^{-1} \sim c_2^{-1}\log n,\, n \to \infty,
$$
satisfy $\sum_{n = k}^{\infty}e^{-\al\La_n^2} < \infty$ for all $\al > 0$. Apply \cite[Lemma 4.5]{MyOwn6}  and complete the proof.
\end{proof}

\begin{lemma}
\label{lemma:cut}
Take a finite, one- or two-sided infinite system $X = (X_n)_{M \le n \le N}$,
with drift coefficients $g_n$ and diffusion coefficients $\si_n^2$, $M \le n \le N$. Here, $M$ and/or $N$ can be infinite. Let $Y = (Y_n)$ be the corresponding system of ranked particles. Take some integers $p, q$ such that $M \le p \le q \le N$. Assume that on some time interval $I \subseteq \BR_+$, we have: 
$$
Y_{p-1}(t) < Y_p(t),\ \ Y_q(t) < Y_{q+1}(t),\ \ t \in I.
$$
Then $(Y_p, \ldots, Y_q)$ behaves as a ranked system of competing Brownian particles with drift coefficients $g_n,\, p \le n \le q$, and diffusion coefficients $\si_n^2,\, p \le n \le q$, on this time interval $I$.
\end{lemma}

\begin{proof} Let $L_{(n, n+1)}$ be the local time of collision between particles $Y_n$ and $Y_{n+1}$. Then $L_{(p-1, p)}$ and $L_{(q, q+1)}$ are constant on $I$. In other words, 
\begin{equation}
\label{eq:d-0}
\md L_{(p-1, p)}(t) = \md L_{(q, q+1)}(t) \equiv 0\ \mbox{on}\ I.
\end{equation}
Recalling Remark~\ref{rmk:ranked-general}, we can rewrite~\eqref{eq:ranked} as 
\begin{equation}
\label{eq:d-Y}
\md Y_n(t) = g_n\,\md t + \si_n\,\md B_n(t) + \frac12\md L_{(n-1, n)} - \frac12\md L_{(n, n+1)}(t),\ p \le n \le q
\end{equation}
Here, $B_n,\, p \le n \le q$, are i.i.d. Brownian motions. Combining~\eqref{eq:d-0} with~\eqref{eq:d-Y}, and using \cite[Propostition 2.2]{MyOwn6}, we complete the proof. 
\end{proof}

Fix an $N < \infty$, and define the wedge
$$
\mathcal W_N := \{y \in \BR^N\mid y_1 \le \ldots \le y_N\}.
$$

\begin{lemma}
\label{lemma:Feller} Take an $N < \infty$. Fix drift and diffusion coefficients $g_k,\, \si_k^2,\, k = 1, \ldots, N$. For every $y \in \mathcal W_N$, denote by $Y^{(y)}$ a process in $\BR^N$ which is the ranked system of $N$ competing Brownian particles with given drift and diffusion coefficients, starting from $Y^{(y)}_n(0) = y_n,\, 1 \le n \le N$. As $x \to y$ in $\mathcal W_N$, we have: $Y^{(x)} \Ra Y^{(y)}$  in $C([0, T], \BR^N)$ for every $T > 0$. 
\end{lemma}

\begin{proof} The system $Y^{(y)}$ is actually an SRBM (semimartingale reflected Brownian motion) in the wedge $\mathcal W_N$, with drift vector $(g_1, \ldots, g_N)$, and covariance matrix $\diag(\si_1^2, \ldots, \si_N^2)$, starting from $y$, see \cite{Williams95}. The statement then follows from the Feller property of SRBM in convex polyhedra from this cited article \cite{Williams95}.
\end{proof}

\begin{rmk} Let us return to a finite system of $N$ competing Brownian particles with drift coefficients $g_1, \ldots, g_N$ and unit diffusion coefficients. Under the assumption~\eqref{eq:stability-intro}, the stationary gap distribution has the product-of-exponentials form given in~\eqref{eq:product-form}. 
Note that the sequence of numbers $\mu_k,\, k = 1, \ldots, N-1$, satisfy the following finite difference equation boundary value problem:
$$
\frac12\mu_{k-1} - \mu_k + \frac12\mu_{k+1} = g_{k+1} - g_k,\ k = 1, \ldots, N-1,
$$
with the following boundary conditions: $\mu_0 = \mu_N = 0$. The solution to this boundary value problem is unique. Moreover, we can represent
\begin{equation}
\label{eq:difference-mu}
\mu_k - \mu_l = 2\left(g_{l+1} + \ldots + g_k\right) - 2(k-l)\ol{g}_N,\, 1 \le l < k < N.
\end{equation}
\label{rmk:BVP}
\end{rmk}

\section*{Acknowledgements}


This research was partially supported by  NSF grants DMS 1007563, DMS 1308340, DMS 1405210, and DMS 1409434.

\end{document}